\documentclass[journal,twoside,web]{ieeecolor}

    \usepackage[pdftex]{graphicx}
    \DeclareGraphicsExtensions{.pdf,.jpeg,.png}
\usepackage{generic}
\usepackage{gensymb}
\usepackage{amssymb}
\usepackage{amsmath}
\usepackage{float}
\usepackage{amsfonts, bm}
\usepackage{color}
\usepackage[noadjust]{cite}

\newcommand{\Set}[1]{\mathbb{#1}}
\newcommand{\System}[1]{\mathcal{#1}}
\newcommand{\Tfs}[1]{\hat{#1}}
\newcommand{\Tfz}[1]{\tilde{#1}}
\newcommand{\Irt}[1]{#1}
\newcommand{\Irk}[1]{\bar{#1}}

\newcommand{\Complexes}{\Set{C}}

\DeclareMathOperator*{\trace}{tr}

\DeclareMathOperator*{\tr}{\text{tr}}

\DeclareMathOperator*{\esssup}{ess sup}

\newtheorem{lemma}{Lemma}
\newtheorem{theorem}{Theorem}
\newtheorem{proposition}{Proposition}
\newtheorem{claim}{Claim}
\newtheorem{corollary}{Corollary}
\newtheorem{definition}{Definition}
\newtheorem{rem}{Remark}
\newtheorem{assumption}{Assumption}

\title{Optimal $\Set{H}_2$ Control with Passivity-Constrained Feedback: Convex Approach %
\thanks{This work was supported by National Science Foundation (NSF) award CPS-2206018. Views expressed are those of the authors and do not necessarily reflect those of the NSF.}
}

\author{
J.T. Scruggs
\thanks{The author is with the Departments of Civil and Environmental Engineering, and Electrical and Computer Engineering, at the University of Michigan, Ann Arbor, MI 48019.  email: {\tt jscruggs@umich.edu} }
}

\begin{document} 

\maketitle

\begin{abstract}
We consider the $\Set{H}_2$-optimal feedback control problem, for the case in which the plant is passive with bounded $\Set{L}_2$ gain, and the feedback law is constrained to be output-strictly passive.
We show that this problem distills to a convex, infinite-dimensional optimal control problem, in which the optimization domain is the Youla parameter for the closed-loop system.
We devise truncated, finite-dimensional optimizations to find sub-optimal controllers, and lower bounds on the optimal objective.
Furthermore we show that both these optimizations converge to the optimal objective of the original infinite-dimensional problem as their respective domains are increased. 
The idea is demonstrated on a simple vibration suppression example. 
\end{abstract}

\section{Introduction}

We consider a class of linear feedback control problems, illustrated in Figure \ref{block_diagram}.
The plant $\System{P}$ constitutes the linear time-invariant (LTI) mapping $\{w,u\} \mapsto \{z,y\}$ and is assumed to be in $\Set{H}_\infty$. 
The channel $u \mapsto y$ is assumed to be passive.
We seek the feedback controller $\System{K}$ that minimizes the standard $\Set{H}_2$  objective for closed-loop mapping $w \mapsto z$, over the set $\Set{K}$ of all linear output-strictly-passive (OSP) mappings.

This problem is motivated by technological design constraints that arise in several areas of engineering, including vibration suppression and robotics. 
A vast array of technologies exist for imposition of feedback in mechanical systems, for the purpose of response suppression.
The simplest technologies are comprised of entirely of passive components, for which the imposed ``feedback'' is physically realized directly by the constitutive laws of these components.
In many circumstances passive control technology is the most practical and preferred, as it may be implemented without the need of a power supply or electronics of any kind.
Furthermore, the closed-loop negative feedback connection of two systems is bounded-input-bounded-output stable, if the plant has finite $\Set{L}_2$ gain (i.e., is in $\Set{H}_\infty$) and is passive, and if the feedback law is OSP \cite{van2000l2}. 
As such, even if considerable uncertainty exists in $\System{P}$, if $u \mapsto y$ is known to be passive with finite gain, then an OSP-constrained $\System{K}$ is guaranteed to be stability-robust. 

Passive control technology has been widely applied in mechanical and structural engineering for over a century \cite{jp1985,sun1995,spencer2003,kelly1986,housner1997}. %\cite{frahm1911,hart1992,hart1993,fujinami1991,yamaura1993,abdel1984,chang1980,kelly1986,jp1985,zuo2005,spencer2003}. 
Passive electrical networks have also been embedded into mechanical vibratory systems for the purpose of response suppression, as electrical shunt circuits connected to transducers \cite{hagood1990,moheimani2003,behrens2005}.
Passive networks can be optimized via network synthesis, both for electrical \cite{anderson2013} as well as mechanical realizations \cite{smith2002}.
It is a classical result \cite{darlington1984} that a multiport LTI network is passive if and only if the matrix transfer function characterizing its driving-point admittance (i.e., the transfer function for channel $u\mapsto y$) is Positive Real (PR).
Furthermore, any rational PR admittance can be realized via a finite network of elementary linear passive components (i.e., inductors, capacitors, resistors, and transformers for electrical networks, and springs, dashpots, inerters, and levers for mechanical networks).
The classical passive network synthesis methods of Brune \cite{brune1931}, Bott and Duffin \cite{bott1949}, and Darlington \cite{Darlington1931} can be used to find such a network realization \cite{anderson2013}.
Further restriction to networks of dissipative passive devices implies further restriction of the feasibility domain to Strict Positive Real (SPR) rational transfer functions \cite{Taylor1974}.

Active (i.e., externally-powered) sensors and actuators allow for more sophisticated control laws, nonlocal measurement feedback, real-time adaptation, and higher control authority.
In choosing a feedback technology (e.g., active vs passive) for an application, the increased financial and logistical demands associated with active systems must be justified by a tangible improvement in the performance objective they achieve relative to passive systems. 
For this justification to be convincing, the feedback law designed for each candidate technology should be optimized to render the best performance achievable, constrained to that particular technology's physical limitations.
Feedback laws for active and passive technologies, when respectively optimized, may turn out to be rather dissimilar. 

\begin{figure}
\centering
\includegraphics[scale=.8]{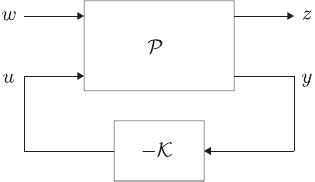}
\caption{Block diagram for feedback system under consideration}
\label{block_diagram}
\end{figure}

Early work related to this idea emerged in the early 1990s in the area of vibration suppression of aerospace structures, by MacMartin and Hall \cite{macmartin1991control,macmartin1991structural}. 
The motivation was to suppress vibrations in large-scale, lightweight space structures.
The heuristic strategy was to achieve favorable vibration suppression performance by optimizing (i.e., maximizing) worst-case net power absorption at the control actuators. evaluated over an uncertainty domain of possible disturbance spectra.
(Because the resultant optimal $\System{K}$ is energy-absorptive at all frequencies, its transfer function $\Tfs{K}(s)$ is guaranteed to be PR.) 
The authors exploited passivity of the plant to frame the problem equivalently as an unconstrained $\Set{H}_\infty$ optimal control problem.
Although an elegant solution was found, the resultant performance was not always favorable.
This is due to the fact that the feedback controller not only influences the power absorbed by the actuators, but also indirectly influences the power injected into the plant by disturbances. 
As such, a controller designed to maximize energy absorption may achieve optimality primarily by inducing the disturbance to inject more power into the plant than in the open-loop case. 

In contrast to this early work, the vast majority of the research on this subject has treated the passivity of $\System{K}$ as a constraint, to be imposed on the optimization of some other objective. 
Particular attention has focused on the optimization of $\System{K}$ for an $\Set{H}_2$ performance objective, with associated Laplace-domain transfer function $\Tfs{K}$ constrained to be strictly positive-real (SPR).
Lozano-Leal and Joshi \cite{lozano1988} introduced a technique which, by appropriate choice of the weights in the $\Set{H}_2$ objective, guarantees to produce a SPR feedback law without the need to impose the constraint explicitly.
The technique has been extended in other studies, including \cite{haddad1994}, which extends it to mixed $\Set{H}_2$/$\Set{H}_\infty$ problems. 
However, the problem solution in \cite{lozano1988} arises only in special circumstances, beyond which the SPR-constrained $\Set{H}_2$ problem has no (known) closed-form solution.
Indeed, it is even unclear under which circumstances the optimal $\Tfs{K}$ is rational \cite{damaren2006}.

Studies by Gapski and Geromel \cite{geromel1997} and Shimomura \cite{shimomura2000,shimomura2002} and more recently Forbes \cite{forbes2013}, used Linear Matrix Inequality (LMI) techniques to develop computationally-efficient design approaches for SPR-constrained 
$\Set{H}_2$ optimal control problems.
In all these approaches, the problem is made tractable by exploiting the Certainty-Equivalence principle, and assuming either the Kalman-Bucy filter (in the case of \cite{geromel1997,shimomura2000,shimomura2002}) or the certainty-equivalent full-state feedback gain (as in \cite{forbes2013}) is the same as for the unconstrained $\Set{H}_2$ problem.
This results in a reduction in the size and complexity of the optimization but also introduces conservatism, as it is not clear that the original problem exhibits certainty-equivalence, or indeed separation of any kind.
Even with conservatism introduced, the optimizations are still nonconvex.
Convexity may be achieved via the use of a common Lyapunov variable, which introduces further conservatism \cite{geromel1997}.
%Alternatively, the reduced nonconvex problem can be solved for a local optimum, e.g., via the use of iterative convex over-bounding \cite{shimomura2000,shimomura2002}.

The SPR-constrained optimization of $\System{K}$ under the $\Set{H}_2$ objective can be framed as a bilinear matrix inequality (BMI) problem.
This class of nonconvex optimization problems is computationally expensive but can be solved using various techniques.
In \cite{papageorgiou2006positive} the BMI problem is solved in YALMIP for a local minimum.
In \cite{warner2015}, the same basic problem is approached, and is solved using iterative convex overbounding techniques. 
%In that study it is noted that the BMI formulation has a further disadvantage, beyond being computationally expensive.
%Specifically, for a given $\Tfs{K}$ its optimization domain contains an infinite number of equivalent realizations, all of which yield the same performance.
%A methodology is used in \cite{warner2015} to eliminate equivalent realizations, with the resultant technique  involving products of five matrix variables.
%This nonconvex optimization problem is then solved using iterative convex overbounding techniques. 
Other studies have considered the problem of PR and SPR-constrained feedback design, over a domain that is restricted to transfer functions realizable with specified component confgurations.
In  \cite{chen2009restricted}, PR-constrained feedback laws of restricted complexity are optimized, and it is shown that the necessary and sufficient conditions for satisfaction of the PR constraint can be simplified over this restricted domain. 

\subsection{Contributions}

This paper revisits the $\Set{H}_2$ optimal control problem with a passivity constraint on $\System{K}$. The primary contributions are:
\begin{enumerate}
\item
We show that if $\System{P}$ is in $\Set{H}_\infty$ and channel $u \mapsto y$ is passive, then the problem is the solution to an infinite-dimensional convex optimization, in the domain of the closed-loop Youla parameter \cite{newton1958, youla1976modern, vidyasagar1985}.
\item
We show that the optimal $\System{K}$ is the limit of a sequence of finite-dimensional convex optimizations, with each successive optimization conducted over a larger-dimensional parametric domain, and rendering in a monotonically-nonincreasing optimial objective.
This result generalizes a technique devised by Scherer \cite{scherer1995,scherer2000},  for multiobjective $\Set{H}_2/\Set{H}_\infty$ problems. 
\item 
Using duality techniques, we show that the dimension of the parametric domain for the $n^{th}$ optimization in this sequence can be reduced from $\mathcal{O}(n^2)$ to $\mathcal{O}(n)$.  
\item 
We devise a second sequence of convex optimizations of growing domain dimensionality, which render a monotonically-increasing lower bound on the objective.  
\item
We show that both optimization sequences converge to the optimal objective of the original infinite-dimensional problem as their respective dimensionalities increase. 
\end{enumerate}

\subsection{Paper organization}

Section \ref{problem_formulation} precisely states the optimization problem to be solved. 
Section \ref{suboptimal_solution} presents a computationally-simple sub-optimal solution to the optimization problem, which is based on certainty-equivalence, and which is used as a reference for the optimal solution.
Section \ref{convex_solution} provides the reformulation of the optimal control problem as an infinite-dimensional convex optimization (Contribution 1). 
Using this result, Section \ref{finite_dimensional_solutions} formulates a finite-dimensional optimization methodology that asymptotically approaches the global optimum as the dimension of the domain is increased (Contributions 2, 3, and 5).
Section \ref{asymptotic_section} develops the optimization sequence for the lower bound on performance (Contributions 4 and 5).
Section \ref{example} provides two examples related of vibration suppression.
Section \ref{conclusions} contains some brief conclusions.

\subsection{Notation}

The sets of real, complex, and integer numbers are respectively denoted $\Set{R}$, $\Set{C}$, and $\Set{Z}$. The sets of positive (nonnegative) reals are  denoted $\Set{R}_{>0}$ ($\Set{R}_{\geqslant 0}$), with similar conventions for the sets of negative (nonpositive) real numbers. 
The sets of positive (nonnegative) integers are denoted $\Set{Z}_{>0}$ ($\Set{Z}_{\geqslant 0}$), with similar conventions for the sets of negative (nonpositive) integers. 
The sets of complex numbers with positive (nonnegative) real parts are denoted $\Set{C}_{>0}$ ($\Set{C}_{\geqslant 0}$), with similar conventions for sets with negative (nonpositive) real parts.
The sets of complex numbers with moduli greater than (greater than or equal to) one are denoted $\Set{O}_{>1}$ ($\Set{O}_{\geqslant 1}$), with similar conventions for sets with moduli less than (less than or equal to) one. 

For a vector $x \in \Set{C}^n$, $\| x \|_p$ denotes the usual $p$-norm, for $p \in \Set{Z}_{>0}$.
For the special case with $p = \infty$, we have that $\| x \|_\infty = \max_i |x_i|$.
For a matrix $X \in \Set{C}^{n\times m}$, $\| X \|_p$ denotes the analogous $p$-norm, and $\| X \|$ with no subscript denotes the maximum singular value.
For both vectors and matrices $\Irt{X}$ that are Lebesgue measurable functions of $t \in \Set{R}$, we define 
\begin{align*}
\| X \|_{\Set{L}_p} \triangleq & \left[ \int_{-\infty}^\infty \| X(t) \|_p^p dt \right]^{1/p} \\
\| X \|_{\Set{L}_\infty} \triangleq & \esssup_{t\in\Set{R}} \| X(t) \|_\infty.
\end{align*}
Set $\Set{L}_p$ contains all $X : \Set{R} \rightarrow \Set{C}^{n\times m}$, for some $n,m \in \Set{Z}_{>0}$, for which the associated norm is finite. 
For both vectors and matrices $\Irk{X}$ that are indexed by $k \in \Set{Z}$, we define 
\begin{align*}
\| \Irk{X} \|_{\Irk{\Set{L}}_p} \triangleq & \left[ \sum_{k=-\infty}^\infty \| \Irk{X}(k) \|_p^p \right]^{1/p} \\
\| \Irk{X} \|_{\Irk{\Set{L}}_\infty} \triangleq & \sup_{k\in\Set{Z}} \| \Irk{X}(k) \|_\infty.
\end{align*}
Set $\Irk{\Set{L}}_p$ is the set of all $\Irk{X} : \Set{Z} \rightarrow \Set{C}^{n\times m}$, for some $n,m \in \Set{Z}_{>0}$, for which the associated norm is finite. 

LTI, continuous-time systems (irrespective of the number of inputs $m$ and outputs $n$) are denoted in calligraphic font, i.e., $\System{S}$. 
The associated impulse response is denoted $\Irt{S}(t) \in \Set{C}^{n\times m}$, for all $t \in \Set{R}$.  
The associated Laplace-domain transfer function is denoted $\Tfs{S}(s) \in \Set{C}^{n\times m}$, for all $s \in \Set{C}$.
We denote as $\Tfs{\Set{H}}_2$ the set of all Laplace-domain transfer functions $\Tfs{S}(s)$ which are analytic on $\Set{C}_{>0}$ and such that
\begin{equation*}
\| \Tfs{S} \|_{\Tfs{\Set{H}}_2} \triangleq \sup\limits_{\sigma \in \Set{R}_{>0}}  \left[ \frac{1}{2\pi} \int_{-\infty}^\infty \| \Tfs{S}(\sigma+j\omega) \|_2^2 d\omega \right]^{1/2} < \infty.
\end{equation*}
We denote as $\Tfs{\Set{H}}_\infty$ the set of all Laplace-domain transfer functions $\Tfs{S}(s)$ that are analytic on $\Set{C}_{>0}$, and which satisfy the uniform bound
\begin{equation*}
\| \Tfs{S} \|_{\Tfs{\Set{H}}_\infty} \triangleq \sup\limits_{s \in \Set{C}_{>0}} \| \Tfs{S}(s) \|_\infty < \infty.
\end{equation*}

Discrete-time transfer functions are denoted with tildes, i.e., $\Tfz{S}(q) \in \Set{C}^{n\times m}$ for all $q \in \Set{C}$.
We denote as $\Tfz{\Set{H}}_2$ the set of all discrete-time transfer functions $\Tfz{S}(q)$ which are analytic on $\Set{O}_{>1}$ and for which 
\begin{equation*}
\| \Tfz{S} \|_{\Tfz{\Set{H}}_2} \triangleq 
\sup\limits_{\beta \in \Set{R}_{>0}} \left[ \frac{1}{2\pi} \int_{-\pi}^\pi \| \Tfz{S}(e^{\beta+j\Omega}) \|_2^2 d\Omega \right]^{1/2} < \infty.
\end{equation*}
We denote as $\Tfz{\Set{H}}_\infty$ the set of all discrete-time transfer functions $\Tfz{S}(q)$ which are analytic on $\Set{O}_{>1}$, and for which 
\begin{equation*}
\| \Tfz{S} \|_{\Tfz{\Set{H}}_\infty} \triangleq \sup\limits_{q \in \Set{O}_{>1}} \| \Tfz{S}(q) \|_\infty < \infty.
\end{equation*}
The Markov parameters associated with such a discrete-time transfer function are denoted $\Irk{S}(k) \in \Set{C}^{n\times m}$ for all $k \in \Set{Z}$. 

\section{Problem Statement} \label{problem_formulation}

\subsection{Plant model}

\begin{assumption} \label{plant_model_assumption}
We presume that plant $\System{P} : \{w,u\} \mapsto \{z,y\}$ is characterized by the $n_P$-order LTI transfer functions, i.e., 
\begin{equation}
\begin{bmatrix} \Tfs{P}_{wz} & \Tfs{P}_{uz} \\ \Tfs{P}_{wy} & \Tfs{P}_{uy} \end{bmatrix}
= \left[ \begin{array}{l|ll} A & B_w & B_u \\ \hline C_z & 0 & D_{uz} \\ C_y & 0 & D_{uy} \end{array} \right]
\label{plant-input-output-model}
\end{equation}
where $A \in \Set{C}^{n_P\times n_P}$, $B_w \in \Set{C}^{n_P\times n_w}$, $B_u \in \Set{C}^{n_P\times n_u}$, $C_z \in \Set{C}^{n_z\times n_P}$, $C_y \in \Set{C}^{n_u\times n_P}$, $D_{uz} \in \Set{C}^{n_z\times n_u}$, and $D_{uy} \in \Set{C}^{n_u \times n_u}$ are constant.
We assume the above realization is minimal.
\end{assumption}
\begin{rem}
Note that the mapping $w \mapsto \{z,y\}$ is assumed to be strictly proper. 
This is in contrast to the standard assumptions of the $\Tfs{\Set{H}}_2$ optimal control problem, for which $w \mapsto y$ is proper but not strictly so, i.e., that $D_{wy} D_{wy}^H \succ 0$. 
\end{rem}

\begin{assumption} \label{plant_assumption_Hinf}
Plant $\System{P}$ has bounded $\Set{L}_2$ gain, i.e., there exists a $\varrho \in \Set{R}_{>0}$ such that for all $\{u,w\} \in \Set{L}_2 \times \Set{L}_2$, 
\begin{equation*}
\| y \|_{\Set{L}_2}^2 + \| z \|_{\Set{L}_2}^2 \leqslant \varrho^2 \left( \| w \|_{\Set{L}_2}^2 + \| u \| _{\Set{L}_2}^2 \right).
\end{equation*}
\end{assumption}

\begin{rem}
Assumption \ref{plant_assumption_Hinf} is equivalent to the statement that $\Tfs{P}_{ij} \in \Tfs{\Set{H}}_\infty$ for $\{i,j\} \in \{w,u\}\times \{z,y\}$.
\end{rem}

\begin{assumption} \label{plant_assumption_passivity}
Plant $\System{P}$ is such that mapping $u \mapsto y$ is passive, i.e., there exists a function $\beta : \Set{L}_2 \rightarrow \Set{R}_{>0}$ such that for all $\{w,u\} \in \Set{L}_2 \times \Set{L}_2$, 
\begin{equation*}
\int_{-\infty}^\infty \left[ u^H(t) y(t) + y^H(t) u(t) \right] dt \geqslant -\beta(w).
\end{equation*}
\end{assumption}

\begin{rem}
With Assumption \ref{plant_assumption_Hinf}, Assumption \ref{plant_assumption_passivity} holds if and only if $\Tfs{P}_{uy}$ is PR, i.e., it is analytic on $\Set{C}_{>0}$ and 
\begin{equation*}
\Tfs{P}^H_{uy}(s) + \Tfs{P}_{uy}(s) \succeq 0, \ \ \forall s \in \Set{C}_{>0}.
\end{equation*}
This holds if and only if $\Tfs{P}_{uy}$ is analytic on $\Set{C}_{>0}$ and
\begin{equation*}
\Tfs{P}_{uy}^H(j\omega) + \Tfs{P}_{uy}(j\omega) \succeq 0, \ \ \forall \omega \in \Set{R}.
\end{equation*}
\end{rem}

\begin{definition}\label{plant_domain}
We denote $\Set{P}$ as the set of all plants $\System{P}$ adhering to Assumptions \ref{plant_model_assumption}, 
\ref{plant_assumption_Hinf} and \ref{plant_assumption_passivity}.
\end{definition}

\subsection{Feedback domain}

\begin{definition}\label{pssivity-domain}
For some $\epsilon \in \Set{R}_{>0}$, we denote $\Set{K}_\epsilon$ as the set of all LTI feedback mappings $\System{K} : y \mapsto -u$ that are OSP with this parameter, i.e., which satisfy 
\begin{equation*} 
\int_{-\infty}^\infty \left[ u^H(t) y(t) + y^H(t) u(t) \right] dt + \epsilon \| u \|_{\Set{L}_2}^2 \leqslant 0
\end{equation*}
for all $y \in \Set{L}_2$.
\end{definition} 

\begin{rem}
Constraint $\System{K} \in \Set{K}_\epsilon$ holds if and only if $\Tfs{K}(s)$ is analytic and uniformly bounded for all $s\in\Complexes_{>0}$, and  
\begin{equation}
\Tfs{K}^H(s) + \Tfs{K}(s) - \epsilon \Tfs{K}^H(s) \Tfs{K}(s) \succeq 0, \ \forall s \in \Set{C}_{>0}.
\label{K_constraint}
\end{equation}
Furthermore, this is true if and only if $\Tfs{K}(s)$ is analytic and uniformly bounded on $\Set{C}_{>0}$ and 
\begin{equation*}
\Tfs{K}^H(j\omega) + \Tfs{K}(j\omega) - \epsilon \Tfs{K}^H(j\omega) \Tfs{K}(j\omega) \succeq 0, \ \forall \omega \in \Set{R}.
\end{equation*}
We note that constraint $\System{K} \in \Set{K}_\epsilon$ is more restrictive than the requirement that $\Tfs{K}$ be PR, and is distinct from the requirements of strong-SPR, SPR, and weak-SPR \cite{brogliato2019dissipative}.
\end{rem}

\begin{rem}
We note that $\Set{K}_\epsilon \subset \Tfs{\Set{H}}_\infty$.
Specifically the $\Set{L}_2$ gain of all $\System{K} \in \Set{K}_\epsilon$ is bounded from above by $\tfrac{2}{\epsilon}$ \cite{van2000l2}.
%Consequently criterion $\System{K} \in \Set{K}_\epsilon$ is not necessary for ideal physical passive devices which do not lie in $\Tfs{\Set{H}}_\infty$, and is violated, for example, by networks of ideal inerters and springs.  However, it can be argued that any realistic passive control technology exhibits bounded colocated input/output behavior, assuming it does not contain rigid body modes. 
\end{rem}

\subsection{Problem statement} \label{problem_statement}

The closed-loop system mapping $\System{T}_{wz} : w \mapsto z$ is characterized by transfer function $\Tfs{T}_{wz}$, as
\begin{align*}
\Tfs{T}_{wz} 
&= \Tfs{P}_{wz} - \Tfs{P}_{uz} \Tfs{K} [ I + \Tfs{P}_{uy} \Tfs{K} ]^{-1} \Tfs{P}_{wy}.
\end{align*}
The feedback-passivity-constrained $\Tfs{\Set{H}}_2$ optimal control problem can be framed as the following optimization:
\begin{equation*} \label{OP1}
\text{OP1} : \left\{ \begin{array}{ll}
\text{Given:} & \System{P} \in \Set{P}, \epsilon \in \Set{R}_{>0} \\
\text{Minimize:} & \gamma \triangleq \| \Tfs{T}_{wz} \|_{\Tfs{\Set{H}}_2}^2 \\
\text{Domain:} & \System{K} \in \Set{K}_\epsilon.
\end{array} \right.
\end{equation*}

\section{A Suboptimal Solution} \label{suboptimal_solution}

This section presents a sub-optimal synthesis technique for OP1.  
It has similarities with the technique proposed in \cite{geromel1997}, in that it assumes a certainty-equivalent controller, which reduces the control design problem to the determination of an associated Luenberger observer gain. 
However, the technique described here differs from \cite{geromel1997}, most notably by the fact that the associated optimization is feasible for all $\System{P} \in \Set{P}$. 
While the technique has only minor novelty in relation to the extant literature, its
inclusion is necessary because it will be used later in the paper to generate a ``reference'' controller, for use in determination of the true optimum. 

The approach is predicated on the use of a certainty-equivalent controller of the form
\begin{equation} \label{49475869797869578}
\Tfs{K} = \left[ \begin{array}{c|c}
A+B_uC_K+B_K(C_y+D_{uy}C_K) & -B_K \\ \hline
-C_K & 0 
\end{array} \right]
\end{equation}
where $C_K$ is a full-state feedback gain and $B_K$ is a full-order observer gain.  
In this context the technique relies on Theorems \ref{4949576070748303} and \ref{59698789585t8gghgi8} below.
Proofs of both are standard, and are omitted from the appendix.

\begin{theorem} \label{4949576070748303} \cite{green1994robust}
Let $\System{P} \in \Set{P}$ and assume $R \triangleq D_{uz}^HD_{uz} \succ 0$ and that the associated Hamiltonian matrix
\begin{equation*}
\begin{bmatrix} A-B_uR^{-1}D_{uz}^HC_z & B_u R^{-1} B_u^H \\ -[C_z^H \left( I - D_{uz} R^{-1} D_{uz}^H \right) C_z] & -[A-B_u R^{-1} D_{uz}^HC_z]^H \end{bmatrix} 
\end{equation*}
has no pure-imaginary eigenvalues. 
Let $\System{K}$ be have a transfer function parametrized as in \eqref{49475869797869578}, where
\begin{equation*}
C_K = -R^{-1} [B_u^HJ+D_{uz}^HC_z]
\end{equation*}
and $J=J^H$ is the stabilizing solution to Riccati equation
\begin{align} \label{49457686968494849495856}
0 =& A^H J + J A + C_z^H C_z - C_K^H R C_K.
\end{align}
Let $B_K$ be an observer gain satisfying $W =W^H \succeq 0$, where 
\begin{equation} \label{394949586869695845}
0 = [A+B_KC_y] W+W [A+B_KC_y]^H + B_wB_w^H
\end{equation}
Then 
\begin{equation} \label{4f9fg99ru4hg0ghg}
\gamma = \gamma_u + \tr( C_K W C_K^H R )
\end{equation}
where $\gamma_u \triangleq \tr( B_w^H J B_w )$.
\end{theorem}

\begin{theorem} \label{59698789585t8gghgi8} \cite{brogliato2019dissipative}
For $\System{K}$ as in \eqref{49475869797869578}, $\System{K}\in\Set{K}_\epsilon$ if and only if there exists $Z=Z^H\succ0$ such that $ZB_K = C_K^H$ and 
\begin{equation} \label{5959668679779699585}
ZA_Z + A_Z^HZ + C_K^H C_Z+ C_Z^H C_K + \epsilon C_K^H C_K 
\preceq 0
\end{equation}
where $A_Z \triangleq A+B_uC_K$ and $C_Z \triangleq C_y+D_{uy}C_K$.
\end{theorem}

From Theorems \ref{4949576070748303} and \ref{59698789585t8gghgi8} it follows that $Z$ can be optimized over the domain characterized by inequality \eqref{5959668679779699585},
 to minimize $\gamma$ as in \eqref{4f9fg99ru4hg0ghg}.
 Auxiliary variable $B_K=Z^{-1}C_K^H$, and auxiliary variable $W$ is found from $B_K$ via \eqref{394949586869695845}.
However, this optimization is nonconvex, and there is no obvious way to recover convexity without the introduction of conservatism.
The following proposition provides one such approach.

\begin{proposition} \label{485696879790786}
Let $\System{P}\in\Set{P}$ and $\epsilon \in \Set{R}_{>0}$, and assume the conditions of Theorem \ref{4949576070748303} hold.
Define 
\begin{align*}
G_p \triangleq & [\epsilon I+D_{uy}+D_{uy}]^{1/2} C_K + [\epsilon+D_{uy}+D_{uy}]^{-1/2} C_y \\
G_n \triangleq & [\epsilon I+D_{uy}+D_{uy}]^{-1/2} C_y.
\end{align*}
Let $Z_0=Z_0^H\succ 0$ satisfy Lyapunov inequality
\begin{align}
\label{4494868676797u}
Z_0 A_Z + A_Z^H Z_0 + G_p^H G_p \preceq & 0
\end{align}
where the associated matrix $A_C \triangleq A + Z_0^{-1} C_K^H C_y$ is Hurwitz.
Let $X_0=X_0^H \succ 0$ satisfy 
\begin{equation}
A_C X_0 + X_0 A_C^H + B_wB_w^H \preceq 0.
\label{4948567686878797}
\end{equation}
Then for any $\Psi=\Psi^H \succ0$, there exist $X=X^H\succ 0$ and $Y=Y^H\succ 0$ satisfying 
\begin{align} 
\label{0303838576869687484859}
\begin{bmatrix} \left( \begin{array}{c} A_C[X-X_0]+ [X-X_0]A_C^H \\ + [Y-Z_0^{-1}] C_K^H C_y X_0 \\
+ X_0 C_y^HC_K[Y-Z_0^{-1}] 
 \end{array} \right)& \star & \star \\
C_K[Y-Z_0^{-1}] & -\Psi^{-1} & \star \\ 
C_y[X-X_0] & 0 & -\Psi \end{bmatrix} &
\nonumber \\ 
\preceq 0 &
\\
\label{4949377384950960687659}
\begin{bmatrix}
\left( \begin{array}{c}
A_Z Y + Y A_Z^H  + Z_0^{-1} G_n^H G_n Z_0^{-1}
\\ - Y G_n^H G_n Z_0^{-1}  - Z_0^{-1} G_n^H G_n Y
\end{array} \right) & \star \\
G_p^H Y & -I \end{bmatrix} \preceq  0 &
\end{align}
where $\star$ denotes Hermitian symmetry.
For feasible $X$ and $Y$, let $\System{K}$ be defined by its transfer function as in \eqref{49475869797869578} with $B_K = YC_K^H$. Then $\System{K}\in\Set{K}_\epsilon$ and satisfies objective bound 
\begin{equation*}
\gamma \leqslant \gamma_u + \tr( C_K X C_K^H R ).
\end{equation*}
\end{proposition}

We thus arrive at a conservative but convex optimization for designing a sub-optimal $\System{K}\in\Set{K}_\epsilon$, as
\begin{equation*} \label{OP2}
\text{OP2} : \left\{ \begin{array}{ll}
\text{Given:} & \System{P} \in \Set{P}, \epsilon \in \Set{R}_{>0} , \Psi=\Psi^H\succ 0 \\
\text{Minimize:} & \tr( C_K X C_K^H R ) \\
\text{Domain:} & X=X^H, Y=Y^H \\
\text{Constraints:} & \eqref{0303838576869687484859},  \eqref{4949377384950960687659}, X \succ 0, Y \succ 0
\end{array} \right.
\end{equation*}

\begin{rem}
Values for $Z_0$ and $X_0$ which satisfy the conditions of Proposition \ref{485696879790786} are in general not unique. 
One convenient approach is find $Z_0$ as the solution to \eqref{4494868676797u}
and find $X_0$ as the solution to \eqref{4948567686878797}, both tightened to equalities. 
If $(A_Z,G_p)$ is not observable then the term $G_p^H G_p$ in \eqref{4494868676797u} must be positively perturbed to assure $Z_0 \succ 0$.  
Similarly, if $(A_C,B_w)$ is not controllable then the term $B_wB_w^H$ must be positively perturbed in \eqref{4948567686878797} to assure $X_0 \succ 0$. 
\end{rem}

\begin{rem}
Although OP2 is guaranteed to return a feasible solution for any $\Psi=\Psi^H\succ 0$, the objective $\gamma$ attained by this solution is sensitive to the specific value of $\Psi$ that is chosen. 
OP2 can be iterated with different values of $\Psi$ to improve $\gamma$, by making $\Psi$ an optimization variable as well.  
Supposing $\Psi_k$ is the value of $\Psi$ at iteration $k$, one can replace the term $-\Psi^{-1}$ in \eqref{0303838576869687484859} with the more conservative $-2\Psi_k^{-1} + \Psi_k^{-1} \Psi \Psi_k^{-1}$ to arrive at an optimization that is convex in $\Psi$ as well as the other variables. 
Execution of OP2 with this substitution results in a new value $\Psi = \Psi_{k+1}$, and the objective attained must be no worse than that attained with $\Psi=\Psi_k$, 
As such, this iteration produces a monotonically-nonincreasing objective.
\end{rem}

\section{Optimal Solution} \label{convex_solution}

This section develops a convex approach to solving OP1 without the introduction of conservatism.
First we convert the problem to an equivalent discrete-time system optimization via a bilinear transformation.
Then we establish certain spectral factorization results that play a central role in the solution.
Following this, we establish the associated Youla parameter for the problem.
Finally, we recast OP1 in terms of the Youla parameter and show that the resultant optimization is convex. 

\subsection{Bilinear transformation}

In order to solve OP1, it is convenient to first convert the problem to one in which the norms and domain constraints are in $\Tfz{\Set{H}}_2$ and $\Tfz{\Set{H}}_\infty$ rather than $\Tfs{\Set{H}}_2$ and $\Tfs{\Set{H}}_\infty$. Toward this end, we employ the bilinear transformation
\begin{equation} \label{q_from_s}
q \triangleq \frac{1+2\tau s}{1-2\tau s} 
\end{equation}
for $\tau \in \Set{R}_{>0}$. 
It is straightforward to verify that the above relationship maps $\Set{C}_{>0} \rightarrow \Set{O}_{>1}$. 
The inverse mapping is
\begin{equation*}
s = \frac{1}{2\tau} \frac{q-1}{q+1} 
\end{equation*}
which is analytic in $q$ for all $q \in \Set{O}_{>1}$.
Consequently, for $\{i,j\} \in \{w,u\} \times \{z,y\}$, $\Tfs{P}_{ij}(s)$ is analytic and uniformly bounded for all $s \in \Set{C}_{>0}$ if and only if
\begin{equation*}
\Tfz{P}_{ij}(q) \triangleq \Tfs{P}_{ij}\left( \frac{1}{2\tau} \frac{q-1}{q+1} \right)
\end{equation*}
is analytic and uniformly bounded for all $q \in \Set{O}_{>1}$. 
It is also the case that for $\Omega \in [-\pi,\pi]$, 
\begin{equation*}
\Tfz{P}_{ij}(e^{j\Omega}) = \Tfs{P}_{ij} (j\omega(\Omega))
\end{equation*}
where
\begin{equation*}
\omega(\Omega) = \tfrac{1}{2\tau} \tanh(\tfrac{1}{2} \Omega )
\end{equation*}
and we note that
\begin{align*}
\frac{d}{d\Omega} \omega(\Omega)
=& \frac{1}{\tau} \frac{1}{|1+e^{j\Omega}|^2}.
\end{align*}
As such,  $\System{P} \in \Set{P}$ implies that $\Tfz{P}_{ij} \in \Tfz{\Set{H}}_\infty$ and that 
\begin{equation}
\Tfz{P}_{uy}^H(q) + \Tfz{P}_{uy}(q) \succeq 0, \ \ \forall q \in \Set{O}_{>1},
\label{posreal_discrete_time_1}
\end{equation}
which holds if and only if 
\begin{equation*}
\Tfz{P}_{uy}^H(e^{j\Omega}) + \Tfz{P}_{uy}(e^{j\Omega}) \succeq 0, \ \ \forall \Omega \in [-\pi,\pi].
\end{equation*}

\subsection{Spectral factorization}

We have that $\Tfz{P}_{ij}$ can be expressed as
\begin{align*}
\Tfz{P}_{ij}
=& \left[ \begin{array}{c|c} \Phi & \Gamma_i \\ \hline \Pi_j & \Delta_{ij} \end{array} \right]\end{align*}
where
\begin{align*}
\Phi \triangleq& (I+2\tau A)(I-2\tau A)^{-1} \\
\Gamma_i \triangleq&  2\tau (I-2\tau A)^{-1} B_i \\
\Pi_j \triangleq& 2C_j (I-2\tau A)^{-1} \\
\Delta_{ij} \triangleq& D_{ij} + 2\tau C_j (I-2\tau A)^{-1} B_i.
\end{align*}
In terms of these matrices, the following theorem is a variant of the KYP Lemma. 

\begin{theorem}\label{PRLemma} \cite{brogliato2019dissipative}
Let $\epsilon, \tau \in \Set{R}_{>0}$ and assume $\System{P}$ satisfies Assumptions \ref{plant_model_assumption} and \ref{plant_assumption_Hinf}.  Then $\System{P}$ satisfies Assumption \ref{plant_assumption_passivity} if and only if there exists a matrix $\Sigma_P =\Sigma_P^H \succeq 0$, and compatible matrices $\Pi_S$ and $\Delta_S$, such that 
\begin{multline*}
\begin{bmatrix} \Phi^H \Sigma_P \Phi - \Sigma_P & \Phi^H \Sigma_P \Gamma_u - \Pi_y^H \\ 
\Gamma_u^H \Sigma_P \Phi - \Pi_y & \Gamma_u^H \Sigma_P \Gamma_u-\Delta_{uy} - \Delta_{uy}^H - \epsilon I \end{bmatrix} \\ = -\begin{bmatrix} \Pi_S^H \\ \Delta_S^H \end{bmatrix} \begin{bmatrix} \Pi_S & \Delta_S \end{bmatrix}.
\end{multline*}
Furthermore, for
\begin{equation*}
\Tfz{S} \triangleq \left[ \begin{array}{c|c} \Phi_S & \Gamma_S \\ \hline \Pi_S & \Delta_S \end{array} \right]
\end{equation*}
with $\Gamma_S = \Gamma_u$ and $\Phi_S = \Phi$, it follows that $\Tfz{S}, \Tfz{S}^{-1} \in \Tfz{\Set{H}}_\infty$ and 
\begin{equation*}
\Tfz{P}_{uy}(q) + \Tfz{P}_{uy}^H(q) + \epsilon I = \Tfz{S}^H(q) \Tfz{S}(q), \ \ \forall q \in \Set{O}_{>1}.
\end{equation*}
\end{theorem}

Our analysis also employs the following related factorization result, which is straight-forward to verify.

\begin{corollary}\label{PRLemma_Corr}
For the conditions of Theorem \ref{PRLemma} we have that there exists a $\Tfz{F}$ such that $\Tfz{F}, \Tfz{F}^{-1} \in \Tfz{\Set{H}}_\infty$, and 
\begin{equation*}
[ \Tfz{P}_{uy}(q) + \Tfz{P}_{uy}^H(q) + \epsilon I ]^{-1} = \Tfz{F}^H(q) + \Tfz{F}(q), \ \ \forall q \in \Set{O}_{>1}.
\end{equation*}
Specifically, 
\begin{equation*}
\Tfz{F} = \left[ \begin{array}{c|c} \Phi_F & \Gamma_F \\ \hline \Pi_F & \Delta_F \end{array} \right]
\end{equation*}
with
\begin{align*}
\Phi_F =& \Phi_S - \Gamma_S \Delta_S^{-1} \Pi_S \\
\Gamma_F =& \left[ \Gamma_S\Delta_S^{-1} - \Phi_F \Sigma_F \Pi_S^H \right] \Delta_S^{-H} \\
\Pi_F =& -\Delta_S^{-1} \Pi_S \\
\Delta_F =& \tfrac{1}{2} \Delta_S^{-1} \left[ I + \Pi_S \Sigma_F \Pi_S^H \right] \Delta_S^{-H}
\end{align*}
and where $\Sigma_F = \Sigma_F^H$ is the solution to 
\begin{equation*}
\Sigma_F = \Phi_F \Sigma_F \Phi_F^H + \Gamma_S\Delta_S^{-1}\Delta_S^{-H} \Gamma_S^H.
\end{equation*}
\end{corollary}

\subsection{Youla parametric domain}

Let
\begin{align*}
\Tfz{T}_{wz}(q) \triangleq & \Tfs{T}_{wz}\left( \frac{1}{2\tau} \frac{1-q}{1+q} \right) \\
\Tfz{K}(q) \triangleq & \Tfs{K} \left( \frac{1}{2\tau} \frac{1-q}{1+q} \right).
\end{align*}
Similarly to the observations above, we have that $\Tfz{T}_{wq}, \Tfz{K} \in \Tfz{\Set{H}}_\infty$ if and only if $\Tfs{T}_{wz}, \Tfs{K} \in \Tfs{\Set{H}}_\infty$.  
Furthermore, we have that
\begin{equation*}
\Tfz{T}_{wz} = \Tfz{P}_{wz} - \Tfz{P}_{uz} \Tfz{K} [I+\Tfz{P}_{uy} \Tfz{K} ]^{-1} \Tfz{P}_{wy}
\end{equation*}
and for $\Tfs{T}_{wz} \in \Tfs{\Set{H}}_2$, the objective of OP1 is
\begin{align}
\gamma
=& \frac{1}{2\pi} \int_{-\infty}^\infty \trace\{ \Tfs{T}_{wz}^H(j\omega) \Tfs{T}_{wz}(j\omega) \} d\omega \nonumber
\\
=& \frac{1}{2\pi} \int_{-\pi}^\pi \trace\{ \Tfz{T}_{wz}^H(e^{j\Omega}) \Tfz{T}_{wz}(e^{j\Omega}) \} \frac{d\omega(\Omega)}{d\Omega} d\Omega \nonumber
\\
=& \frac{1}{\tau} \| \Tfz{T}_{fz} \|_{\Tfz{\Set{H}}_2}^2
\label{Tfz_norm}
\end{align}
where
\begin{align*}
\Tfz{T}_{fz}
=& \Tfz{P}_{fz} - \Tfz{P}_{uz} \Tfz{K} [ I + \Tfz{P}_{uy} \Tfz{K} ]^{-1} \Tfz{P}_{fy} 
\end{align*}
with 
\begin{align*}
\Tfz{P}_{fz}(q) =& \frac{1}{1+q} \Tfz{P}_{wz}(q), & 
\Tfz{P}_{fy}(q) =& \frac{1}{1+q} \Tfz{P}_{wy}(q).
\end{align*}
Furthermore, it is straightforward to verify that 
\begin{align*}
\begin{bmatrix} \Tfz{P}_{fz} \\ \Tfz{P}_{fy} \end{bmatrix}
=& \left[ \begin{array}{c|c} \Phi & \Gamma_f \\ \hline \Pi_z & 0 \\ \Pi_y & 0 \end{array} \right]
\end{align*}
where
\begin{equation*}
\Gamma_f = \tau B_w
\end{equation*}

Let the Youla parameter $\Tfz{Q}$ be defined as
\begin{equation}
\Tfz{Q} \triangleq \Tfz{K} [I+\Tfz{P}_{uy}\Tfz{K}]^{-1}
\label{Q_from_K}
\end{equation}
with the inverse relationship being
\begin{equation}
\Tfz{K} = [I-\Tfz{Q}\Tfz{P}_{uy}]^{-1} \Tfz{Q}.
\label{K_from_Q}
\end{equation}
Then we have that 
\begin{equation}
\Tfz{T}_{fz} = \Tfz{P}_{fz} - \Tfz{P}_{uz} \Tfz{Q} \Tfz{P}_{fy}
\label{Tfz_from_Q}
\end{equation}
which, from \eqref{Tfz_norm}, implies that the minimization objective of OP1 is quadratic and convex in $\Tfz{Q}$. 
The key observation of this paper is that the domain constraint for OP1 is also convex in $\Tfz{Q}$, as is shown in the following theorem.

\begin{theorem} \label{59595686877005}
Assume $\System{P} \in \Set{P}$ and let $\epsilon, \tau \in \Set{R}_{>0}$. Then $\System{K} \in \Set{K}_\epsilon$ if and only if $\Tfz{Q} \in \Tfz{\Set{H}}_\infty$ and if
\begin{equation}
\Tfz{U} \triangleq \begin{bmatrix} \Tfz{Q} & \Tfz{Q} \\ 0 & \Tfz{F} \end{bmatrix}
\label{4758686748459845}
\end{equation}
satisfies 
\begin{equation}
\Tfz{U}(e^{j\Omega}) + \Tfz{U}^H(e^{j\Omega}) \succeq 0, \ \forall \Omega\in[-\pi,\pi].
\label{Q_constraint}
\end{equation}
\end{theorem}

\subsection{Convex reformulation of OP1 in Youla domain}

We parametrize $\Tfz{Q}(q)$ via its infinite set of Markov parameters $Q(k)$, $k \in \Set{Z}_{\geqslant 0}$, as 
\begin{equation*} 
\Tfz{Q}(q) = \sum\limits_{k=0}^\infty \Irk{Q}(k) q^{-k} 
\end{equation*}
where we note that coefficients $\Irk{Q}(k)$ for $k<0$ must be zero in order for $\Tfz{Q} \in \Tfz{\Set{H}}_\infty$. 
Using the Youla parametrization from the previous subsection and the definitions above, we can reformulate OP1 as  OP3, below. 
\begin{equation*}
\text{OP3} : \left\{ \begin{array}{ll}
\text{Given:} & \System{P} \in \Set{P}, \epsilon, \tau \in \Set{R}_{>0} \\
\text{Minimize:} & \gamma  \\
\text{Domain:} & \Irk{Q}(k), \forall k \in \Set{Z}_{\geqslant 0} \\
\text{Constraint:} & \eqref{Q_constraint}
\end{array} \right.
\end{equation*}
As such, OP3 has a countably-infinite-dimensional domain, a convex quadratic objective on this domain, and uncountably-infinite convex constraints, each of which is a LMI parametrized by $\Omega\in[-\pi,\pi]$. 
Let the optimum objective be denoted $\gamma^\star$, and the associated optimum Markov sequence be denoted $\{ \Irk{Q}^\star(k), k \in \Set{Z}_{\geqslant 0} \}$.

\begin{theorem}\label{uniqueness_theorem}
For the solution to OP3, 
\begin{equation} \label{49485768jgig8uhgh}
\sum\limits_{k=0}^\infty \| \Irk{Q}^\star (k) \|^2 < \infty.
\end{equation}
Moreover, define $\Set{W} \subseteq [-\pi,\pi]$ as 
\begin{multline*}
\Set{W} = \big\{ \Omega \in [-\pi,\pi] \ : \ \Tfz{P}_{uz}^H(e^{j\Omega}) \Tfz{P}_{uz}(e^{j\Omega}) \succ 0, \\ \land \ \Tfz{P}_{fy}(e^{j\Omega}) \Tfz{P}_{fy}^H(e^{j\Omega}) \succ 0 \big\}
\end{multline*}
and let $\Set{W}'$ be its complement. 
Then if $\Set{W}'$ has zero measure, the solution to OP3 is unique.
\end{theorem}

\subsection{Independence of $\tau$}

Regarding the choice of $\tau$, the theorem below establishes that all choices are equivalent, in the sense that they all result in the same $\gamma^\star$ in OP3.

\begin{theorem} \label{45050607007594848586}
Let $\tau_a, \tau_b \in \Set{R}_{>0}$ and let the optimal objectives of OP3 be denoted $\gamma_a^\star$ and $\gamma_b^\star$, and let $\Tfz{Q}_a^\star$ and $\Tfz{Q}_b^\star$ be the corresponding optimal discrete-time Youla parameters. 
 Then $\gamma^\star_a=\gamma^\star_b$.
Furthermore, if the conditions of Theorem \ref{uniqueness_theorem} hold then $ \| \Tfs{Q}_a^\star - \Tfs{Q}_b^\star \|_{\Tfs{\Set{H}}_2} = 0$, where $\Tfs{Q}_a^\star$ and $\Tfs{Q}_b^\star$ are the corresponding continuous-time optimal Youla parameters obtained with \eqref{q_from_s} with $\tau_a$ and $\tau_b$.
\end{theorem}

\section{Truncated Optimal Solutions}  \label{finite_dimensional_solutions}

\subsection{Incremental Youla parameter} \label{incremental}

In principle, OP3 provides a convex equivalent to OP1, the solution of which is the stated goal of this paper.
However, OP3 requires optimization over a countably-infinite domain, and is therefore impractical.
To be implemented, the number of nonzero Markov parameters must be truncated to a finite number, $n$, resulting in a sequence of suboptimal solutions that asymptotically approaches optimality as $n\rightarrow\infty$.  
Practically, evaluation of this sequence is ended when $n$ becomes sufficiently large that  convergence can be established up to a specified tolerance. 
It is consequently advantageous to employ techniques that reduce the value of $n$ that is needed to achieve convergence.

One means of accomplishing this is to solve for the optimal solution as an incremental improvement relative to the suboptimal design from Section \ref{suboptimal_solution}. For clarity, we refer to the bilinear-transformed suboptimal controller described in Section \ref{suboptimal_solution} as $\Tfz{K}_0$, and express $\Tfz{K} =  \Tfz{K}_0 + \Tfz{K}_1$
where $\Tfz{K}_1 \in \Tfz{\Set{H}}_\infty$ is the incremental feedback law to be designed.
Denote $u_1 \triangleq -\System{K}_1 y$.
Then we have that  $\System{K}_0$ can be subsumed into the plant model, resulting in a mapping $\System{P}_1 : \{f,u_1\} \mapsto \{z,y\}$ characterized by 
\begin{align}
\left[ \begin{array}{ll} \Tfz{P}_{1,fz} & \Tfz{P}_{1,u_1z} \\ \Tfz{P}_{1,fy} & \Tfz{P}_{1,u_1y} \end{array} \right]
=&
\begin{bmatrix} \Tfz{P}_{fz}-\Tfz{P}_{uz}\Tfz{Q}_0\Tfz{P}_{fy} & \Tfz{P}_{uz} \Tfz{M}_0 
\\ \Tfz{N}_0 \Tfz{P}_{fy} & \Tfz{N}_0 \Tfz{P}_{uy} 
\end{bmatrix}
\label{V_transfer_functions}
\end{align}
where
\begin{align*}
\Tfz{Q}_0 \triangleq & \Tfz{K}_0 [ I+\Tfz{P}_{uy}\Tfz{K}_0]^{-1} \\
\Tfz{M}_0 \triangleq & [I+\Tfz{K}_0\Tfz{P}_{uy}]^{-1} \\
\Tfz{N}_0 \triangleq & [I+\Tfz{P}_{uy}\Tfz{K}_0]^{-1}.
\end{align*}
With incremental controller $\System{K}_1$ imposed, the closed-loop transfer function characterizing mapping $f \mapsto z$ is
\begin{equation} \label{373484758567}
\Tfz{T}_{fz} = \Tfz{P}_{1,fz} - \Tfz{P}_{1,u_1z} \Tfz{Q}_1 \Tfz{P}_{1,fy}
\end{equation}
where the incremental Youla parameter $\Tfz{Q}_1$ is
\begin{equation*}
\Tfz{Q}_1 = \Tfz{K}_1 [ I + \Tfz{P}_{1,u_1y} \Tfz{K}_1 ]^{-1} .
\end{equation*}
It is straightforward to verify that $\Tfz{Q}$ is related to $\Tfz{Q}_1$ via
\begin{equation}
\label{67967685875959}
\Tfz{Q} = \Tfz{Q}_0 + \Tfz{M}_0  \Tfz{Q}_1 \Tfz{N}_0.
\end{equation}
The constraint $\System{K}\in\Set{K}_\epsilon$ therefore holds if and only if \eqref{Q_constraint} holds with $\Tfz{U}$ as in \eqref{4758686748459845} evaluated with $\Tfz{Q}$ as in \eqref{67967685875959}. 

\subsection{Primal formulation}\label{primal_formulation}

To finite-dimensionalize the optimization domain, we parametrize $\Tfz{Q}_1$ by its first $n+1$ Markov coefficients, assuming the rest are zero, i.e., 
\begin{equation}\label{FIR}
\Tfz{Q}_1(q) = \sum\limits_{k=0}^n \Irk{Q}_1(k) q^{-k}
\end{equation}
for some $n\in\Set{Z}_{>0}$.
For convenience, define $x_n$ as
\begin{equation}\label{x_def}
x_n \triangleq \textrm{vec}\{ \Irk{Q}_1(0),...,\Irk{Q}_1(n) \} .
\end{equation}
For a given $n$ we denote the dependency of the objective on the truncated Youla parameters as $\gamma_n(x_n)$.
It can be evaluated in closed-form, as formalized in the following claim.

\begin{claim} \label{quadratic_performance_objective}
There exist $\gamma_{n,0}\in\Set{R}$, $g_n\in\Set{C}^{(n+1)n_u^2}$, and $H_n=H_n^H\in\Set{C}^{(n+1)n_u^2\times (n+1)n_u^2}$ such that
\begin{equation} \label{484577686874747686}
\gamma_n(x_n) = \gamma_{n,0} + \tfrac{1}{2} \left( g_n^H x_n + x_n^H g_n + x_n^H H_n x_n \right)
\end{equation}
where $H_n \succeq 0$.
Details regarding the derivation of $\gamma_{n,0}$, $g_n$, and $H_n$ are given in the appendix.
\end{claim}

Constraint \eqref{Q_constraint} is infinite-dimensional, but may be replaced by an equivalent finite-dimensional constraint through the use of auxiliary variables.
To present this result, we first note that $\Tfz{U}$ has a state space realization 
\begin{equation} \label{495869669}
\Tfz{U} = \left[ \begin{array}{cc|c}
	 \Phi_{U11} & \Phi_{U12}(x_n) & \Gamma_{U1}(x_n) \\
	 0                & \Phi_{U22}         & \Gamma_{U2}        \\ \hline
	 \Pi_{U1}     & \Pi_{U2}(x_n)     & \Delta_U(x_n)
\end{array} \right] 
\end{equation}
where $\Phi_{U11}$, $\Phi_{U22}$, $\Gamma_{U2}$, and $\Pi_{U1}$ are independent of $x_n$, and $\Pi_{U12}(x_n)$, $\Gamma_{U1}(x_n)$, $\Pi_{U2}(x_n)$ and $\Delta_U(x_n)$ are affine in $x_n$.
In this context, the following variant of the KYP Lemma.  

\begin{lemma} \label{KYP_theorem}
For $\Tfz{U}$ realized as in \eqref{495869669}, it satisfies \eqref{Q_constraint} if and only if there exist compatible matrices $\Sigma_{U1}=\Sigma_{U1}^H \succ 0$, $\Sigma_{U2}=\Sigma_{U2}^H \succ 0$, and $\Lambda_U$ such that 
\begin{align}
&
\begin{bmatrix} 
0 & 0 & 0  \\
\Pi_{U1} \Lambda_U - \Pi_{U2} & -\Delta_U & 0  \\
\Phi_{U12}-\Phi_{U11} \Lambda_U + \Lambda_U \Phi_{U22} & \Gamma_{U1}+\Lambda_U \Gamma_{U2} & 0
\end{bmatrix}^H
\nonumber \\ & 
+
\begin{bmatrix} 
0 & 0 & 0  \\
\Pi_{U1} \Lambda_U - \Pi_{U2} & -\Delta_U & 0  \\
\Phi_{U12}-\Phi_{U11} \Lambda_U + \Lambda_U \Phi_{U22} & \Gamma_{U1}+\Lambda_U \Gamma_{U2} & 0
\end{bmatrix}
\nonumber \\ &
+
\begin{bmatrix} 
0 & 0 & 0 \\
0 & \Pi_{U1} \Sigma_{U1} \Pi_{U1}^H & -\Pi_{U1} \Sigma_{U1} \Phi_{U11}^H  \\
0 & -\Phi_{U11} \Sigma_{U1} \Pi_{U1}^H & \Phi_{U11} \Sigma_{U1} \Phi_{U11}^H-\Sigma_{U1}
\end{bmatrix} 
\nonumber \\ & 
+
\begin{bmatrix} 
\Phi_{U22}^H \Sigma_{U2} \Phi_{U22} - \Sigma_{U2} & \Phi_{U22}^H \Sigma_{U2} \Gamma_{U2} & 0 \\
\Gamma_{U2}^H \Sigma_{U2} \Phi_{U22} & \Gamma_{U2}^H \Sigma_{U2} \Gamma_{U2} & 0 \\
0 & 0 & 0 
\end{bmatrix} 
\preceq 0
\label{convex_pr_lmi}
\end{align}
\end{lemma}

Noting that matrix inequality \eqref{convex_pr_lmi} is linear in $x_n$, $\Sigma_{U1}$, $\Sigma_{U2}$, and $\Lambda_U$, 
we therefore have the following convex, finite-dimensional optimization for the optimal objective, $\gamma_n^\star$, attainable with truncated Youla parameters: 
\begin{equation*}\label{OP4}
\text{OP4} : \left\{ \begin{array}{ll}
\text{Given:} & \System{P} \in \Set{P}, \epsilon, \tau \in \Set{R}_{>0}, n \in \Set{Z}_{>0} \\
\text{Minimize:} & \gamma_n (x_n) \text{ as in } \eqref{484577686874747686} \\
\text{Domain:} & x_n,  \Sigma_{U1} = \Sigma_{U1}^H, \Sigma_{U2} = \Sigma_{U2}^H, \Lambda_U  \\
\text{Constraints:} & \eqref{convex_pr_lmi}, \Sigma_{U1} \succ 0, \Sigma_{U2} \succ 0 
\end{array} \right.
\end{equation*}
The objective in OP4 is convex and quadratic, while the finite-dimensional constraint is a LMI.
Due to its finite-dimensionality and convexity, OP4 can be solved efficiently.

\subsection{Existence and uniqueness}

Clearly OP4 is guaranteed to be feasible, because \eqref{convex_pr_lmi} is satisfied with $x_n=0$.
Convexity and positive-semidefiniteness of the objective and constraint, together with the fact that the problem has guaranteed feasibility, ensure that OP4 has a global minimum, $\gamma_n^\star$. 
However, these facts do not assure that the problem has a unique corresponding minimizer, $x_n^\star$. However, this can be assured by the imposition of extra assumptions on $\System{P}$, as illustrated by the following corollary to Theorem \ref{uniqueness_theorem}, the proof of which is virtually identical. 

\begin{corollary}
For OP4, define $\Set{W}' \subseteq [-\pi,\pi]$ as in Theorem \ref{uniqueness_theorem}.  Then $x_n^\star$ are unique if $\Set{W}'$ has zero measure. 
\end{corollary}

\subsection{Asymptotic convergence} \label{asymptotic}

We consider the asymptotic behavior of $\gamma_n^\star$ and associated minimizers $\{\Irk{Q}_{n}^\star(k), \ k \in \{0,...,n\} \}$ as $n \rightarrow \infty$.
In this context, we have the following convergence theorem.

\begin{theorem} \label{459569960707096069059}
For OP3, let the associated Youla parameter be denoted $\Tfz{Q}^\star$ and the corresponding optimum objective of OP3 be denoted $\gamma^\star$.
Let $\Tfz{Q}_{1,n}^\star$ and $\gamma_n^\star$ be the solution to OP4, and let $\Tfz{Q}_n^\star \triangleq \Tfz{Q}_0 + \Tfz{M}_0 \Tfz{Q}_{1,n}^\star \Tfz{N}_0$ be the corresponding total Youla parameter associated with this solution,. 
Then $\gamma_n^\star$ is monotonically nonincreasing in $n$, and $\lim_{n\rightarrow\infty} \gamma_n^\star = \gamma^\star$.
Furthermore, if the conditions of Theorem \ref{uniqueness_theorem} hold then 
\begin{equation}
\label{76978696585}
\lim\limits_{n\rightarrow\infty} \| \Tfz{Q}^\star - \Tfz{Q}_n^\star \|_{\Tfz{\Set{H}}_2} = 0.
\end{equation}
\end{theorem}

\subsection{Dual formulation} \label{duality_section}

As $n$ is increased, the optimization domain for OP4 increases like $\mathcal{O}(n^2)$ due to the usage of matrix variable $\Sigma_U$.
There is a considerable advantage to the reformulation of OP4 in its dual domain, for which the dimensionality increases linearly with $n$ rather than quadratically.

Introduce Lagrange multiplier matrix $\Upsilon = \Upsilon^H \succeq 0$, partitioned with the same dimensions as in LMI \eqref{convex_pr_lmi}, i.e., 
\begin{equation*}
\Upsilon = \begin{bmatrix} \Upsilon_{11} & \Upsilon_{12} & \Upsilon_{13} \\ \Upsilon_{12}^H & \Upsilon_{22} & \Upsilon_{33} \\ \Upsilon_{13}^H & \Upsilon_{23}^H & \Upsilon_{33} \end{bmatrix}
\end{equation*}
Then OP4 is expressed equivalently as
\begin{equation*}
\min\limits_{x_n, \Sigma_{U1}=\Sigma_{U1}^H, \Sigma_{U2}=\Sigma_{U2}^H, \Lambda_U} \max\limits_{\Upsilon=\Upsilon^H \succeq 0} \gamma_{pd}(x_n,\Sigma_{U1},\Sigma_{U2},\Lambda_U, \Upsilon)
\end{equation*}
where $\gamma_{pd}$ is the Lagrangian (i.e., primal-dual) objective, i.e., 
{\small
\begin{align*}
& \gamma_{pd} = 
\gamma_{n,0} + \tfrac{1}{2} \left( g_n^H x_n + x_n^H g_n + x_n^H H_n x_n \right)
\nonumber \\ & 
+ \tr\left\{ \begin{bmatrix} \Upsilon_{11} & \Upsilon_{12} \\ \Upsilon_{12}^H & \Upsilon_{22} \end{bmatrix} \begin{bmatrix} \left( \begin{array}{c} \Phi_{U22}^H\Sigma_{U2}\Phi_{U22} \\ -\Sigma_{U2} \end{array} \right) & \Phi_{U22}^H\Sigma_{U2}\Gamma_{U2} \\ \Gamma_{U2}^H\Sigma_{U2}\Phi_{U22} & \Gamma_{U2}^H\Sigma_{U2}\Gamma_{U2} \end{bmatrix} \right\}
\nonumber \\ & 
+ \tr\left\{ 
\begin{bmatrix} \Upsilon_{22} & \Upsilon_{23} \\ \Upsilon_{23}^H & \Upsilon_{33} \end{bmatrix}
\begin{bmatrix} \Pi_{U1}\Sigma_{U1}\Pi_{U1}^H & -\Pi_{U1}\Sigma_{U1}\Phi_{U11}^H \\ -\Phi_{U11}\Sigma_{U1}\Pi_{U1}^H & \left( \begin{array}{c} \Phi_{U11} \Sigma_{U1}\Phi_{U11}^H \\ - \Sigma_{U1} \end{array} \right) \end{bmatrix}  
\right\}
\nonumber \\ & 
+ \tr\left\{
\begin{bmatrix} \Upsilon_{12} & \Upsilon_{13} \\ \Upsilon_{22} & \Upsilon_{23} \end{bmatrix}
\begin{bmatrix} \Pi_{U1} \Lambda_U & 0 \\ -\Phi_{U11} \Lambda_U + \Lambda_U \Phi_{U22} & \Lambda_U \Gamma_{U2} \end{bmatrix}
\right\} 
\nonumber \\ & 
+ \tr\left\{
\begin{bmatrix} \Upsilon_{12} & \Upsilon_{13} \\ \Upsilon_{22} & \Upsilon_{23} \end{bmatrix}^H
\begin{bmatrix} \Pi_{U1} \Lambda_U & 0 \\ -\Phi_{U11} \Lambda_U + \Lambda_U \Phi_{U22} & \Lambda_U \Gamma_{U2} \end{bmatrix}^H
\right\} 
\nonumber \\ & 
+ \tr\left\{
\begin{bmatrix} \Upsilon_{12} & \Upsilon_{13} \\ \Upsilon_{22} & \Upsilon_{23} \end{bmatrix}
\begin{bmatrix} -\Pi_{U2}(x_n) & -\Delta_U(x_n) \\ \Phi_{U12}(x_n) & \Gamma_{U1}(x_n) \end{bmatrix}
\right\} 
\nonumber \\ &
+ \tr\left\{
\begin{bmatrix} \Upsilon_{12} & \Upsilon_{13} \\ \Upsilon_{22} & \Upsilon_{23} \end{bmatrix}^H
\begin{bmatrix} -\Pi_{U2}(x_n) & -\Delta_U(x_n) \\ \Phi_{U12}(x_n) & \Gamma_{U1}(x_n) \end{bmatrix}^H
\right\} 
\end{align*}
}%
In order for $\gamma_{pd}$ to be extremal in $\Sigma_{U1}$, $\Sigma_{U2}$, and $\Lambda_U$, the Lagrange multipliers must satisfy 
\begin{align}
\Upsilon_{11} =& \begin{bmatrix} \Phi_{U22} & \Gamma_{U2} \end{bmatrix} \begin{bmatrix} \Upsilon_{11} & \Upsilon_{12} \\ \Upsilon_{12}^H & \Upsilon_{22} \end{bmatrix} \begin{bmatrix} \Phi_{U22}^H \\ \Gamma_{U2}^H \end{bmatrix}.
\label{57676768697978797}
\\
\Upsilon_{33} =& \begin{bmatrix} -\Pi_{U1}^H & \Phi_{U11}^H \end{bmatrix} \begin{bmatrix} \Upsilon_{22} & \Upsilon_{23} \\ \Upsilon_{23}^H & \Upsilon_{33} \end{bmatrix}
\begin{bmatrix} -\Pi_{U1} \\ \Phi_{U11} \end{bmatrix}
\label{494938739595687696}
\\
0 =& -\Upsilon_{12} \Pi_{U1} - \Upsilon_{13} \Phi_{U11} + \Phi_{U22} \Upsilon_{13} + \Gamma_{U2} \Upsilon_{23}
\label{40586767696854485}
\end{align}
Because $\Tfz{U}\in\Tfz{\Set{H}}_\infty$ the moduli of the eigenvalues of $\Phi_{U11}$ and $\Phi_{U22}$ are less than $1$, and consequently \eqref{57676768697978797} and \eqref{494938739595687696} uniquely determine $\Upsilon_{11}$ and $\Upsilon_{33}$ in terms of $\Upsilon_{12}$, $\Upsilon_{22}$, and $\Upsilon_{23}$. 
Equation \eqref{40586767696854485} is a Sylvester equation, and if $\Phi_{U11}$ and $\Phi_{U22}^H$ have no eigenvalues in common, it determines a $\Upsilon_{13}$ uniquely in terms of $\Upsilon_{12}$ and $\Upsilon_{23}$. 
However, due to the specifics of the way the realization for $\tilde{U}$ was constructed, this is generally not the case, and consequently, \eqref{40586767696854485} determines a subspace in which the parameters $\{ \Upsilon_{12}, \Upsilon_{13}, \Upsilon_{22}, \Upsilon_{23} \}$ must lie.

There exist a vector $h_n$ and matrix $E_n$ such that
\begin{equation*}
\textrm{vec} \left\{ \begin{bmatrix} -\Pi_{U2}(x_n) & -\Delta_{U}(x_n) \\ \Phi_{U12}(x_n) & \Gamma_1(x_n) \end{bmatrix} \right\}
= h_n + E_n x_n.
\end{equation*}
Letting 
\begin{equation}
\psi_n = \textrm{vec} \left\{ \begin{bmatrix} \Upsilon_{12} & \Upsilon_{13} \\ \Upsilon_{22} & \Upsilon_{23} \end{bmatrix}^H \right\}
\label{383477578656878}
\end{equation}
we then have that 
\begin{align*}
\gamma_{pd} =& \gamma_{0,n} + \psi_n^H \left( h_n+E_nx_n \right) + \left( h_n+E_nx_n \right)^H \psi_n \nonumber \\
& + \tfrac{1}{2} \left( g_n^H x_n + x_n^H g_n + x_n^H H_n x_n \right).
\end{align*}
The dual objective, denoted $\gamma_{d,n}$, is then the minimization of $\gamma_{pd}$ over $x_n$.
Because the primal problem is known to be feasible, it is known that the dual objective is bounded from above, for all $x_n$. 
Consequently, the null space of $H_n$, if it is nonempty, must be contained in the null spaces of $2\psi_n^H E_n+g_n^H$ for all feasible $\psi_n$. 
In other words, if the range space of some matrix $\eta_H$ spans the null space of $H_n$, then we require that $\psi_n$ satisfy $(2\psi_n^H E_n + g_n^H)\eta_H = 0$. 
As such,
\begin{align}
\gamma_{d,n}(\psi_n) =& \gamma_{0,n} + \psi_n^H h_n + h_n^H \psi_n \nonumber \\ 
&  - \tfrac{1}{2} (2E_n^H\psi_n+g_n)^H H_n^{\dagger} (2E_n^H\psi_n+g_n) 
\label{57686978695858}
\end{align}
where $(\cdot)^\dagger$ denotes the Moore-Penrose inverse.

We thus arrive at the dual formulation of OP4, as
\begin{equation*}\label{OP5}
\text{OP5} : \left\{ \begin{array}{ll}
\text{Given:} & \System{P} \in \Set{P}, \epsilon, \tau \in \Set{R}_{>0}, n \in \Set{Z}_{>0} \\
\text{Maximize:} & \gamma_{d,n}(\psi_n) \text{ as in } \eqref{57686978695858} \\
\text{Domain:} & \psi_n \text{ as in }\eqref{383477578656878} \\
\text{Constraints:} & \Upsilon(\psi_n) \succeq 0, \eqref{40586767696854485}, \\
&   (2\psi_n^H E_n + g_n^H)\eta_H = 0
\end{array} \right.
\end{equation*}
where $\Upsilon(\psi_n)$ is characterized uniquely by \eqref{57676768697978797}, \eqref{494938739595687696}, and \eqref{383477578656878}.
The advantage of OP5 comes from the fact that the dimension of $\Upsilon_{13}$ and $\Upsilon_{23}$ grow linearly with $n$, while the dimensions of $\Upsilon_{12}$ and $\Upsilon_{22}$ are independent of $n$.  Consequently, the optimization domain scales like $\mathcal{O}(n)$ rather than $\mathcal{O}(n^2)$ as with OP4. 

Let $\psi_n^\star$ be a maximizer of OP5. 
If the conditions of Theorem \ref{uniqueness_theorem} hold then the parameter $x_n^\star$ corresponding to $\psi_n^\star$ must be unique.
If this is the case it must be that $H_n \succ 0$, and therefore invertible. 
The corresponding primal optimal parameters is obtained from the dual solution via
\begin{equation*}
x_n^\star = -H_n^{-1} \left( 2E_n^H\psi_n^\star + g_n \right).
\end{equation*}

\section{Lower Bound on Optimallity} \label{asymptotic_section}

The truncated solution technique in the previous section results in an incremental Youla parameter $\Tfz{Q}_1$ which is optimal, over the domain of all transfer functions for which only the first $n+1$ Markov parameters (with a given bilinear parameter $\tau$) are nonzero.
Theorem \ref{459569960707096069059} assures that as $n\rightarrow\infty$, the minimized objective $\gamma_n^\star$ obtained with this technique monotonically converges from above to $\gamma^\star$, the solution to the infinite dimensional problem OP3.  
Theorem \ref{45050607007594848586} establishes that $\gamma^\star$ is also the optimal objective of original problem OP1, independently of $\tau \in \mathbb{R}$.
In this section we devise a second finite-dimensional optimization algorithm, which finds a lower bound on $\gamma^\star$. 
This is important because it allows for us to ascertain a finite $n$ for which OP4 (or equivalently OP5) solves OP1, up to a desired tolerance. 

For $n \in \mathbb{Z}_{>0}$, define
\begin{equation}
\label{4885ghjgig9guhif8gvugvhjvigfiu}
\check{\gamma}_n \triangleq \frac{1}{\tau} \sum\limits_{k=0}^n \tr\left( \bar{T}_{fz}^H(k) \bar{T}_{fz}(k) \right).
\end{equation}
It is an immediate consequence of Plancharel's theorem that for all $\System{P} \in \Set{P}$, $\Tfz{Q} \in \Tfz{\Set{H}}_2$, and $n \in \mathbb{Z}_{>0}$, $\check{\gamma}_n \leqslant \gamma$. 
Due to causality of $\Tfz{P}$ and $\Tfz{Q}$, it follows that $\Irk{T}_{fz}(k)$ depends on $\Irk{Q}(\ell)$ for $\ell \in \{0,...,k\}$ and so $\check{\gamma}_n$ is independent of $\Irk{Q}(k)$ for $k>n$. 

Next we introduce a relaxation of constraint \eqref{Q_constraint}, as the following constraint:
\begin{equation}
\label{V_constraint}
\Tfz{V}(e^{j\Omega}) + \Tfz{V}^H(e^{j\Omega}) \succeq 0, \ \forall \Omega \in [-\pi,\pi]
\end{equation}
where the first $n+1$ Markov coefficients of $\Tfz{V}$ are required to have the form
\begin{equation*}
\Irk{V}(k) = \begin{bmatrix} \Irk{Q}(k) & 0 \\ \Irk{Q}(k) & \Irk{F}(k) \end{bmatrix}, \ \forall k \in \{0,...,n\}
\end{equation*}
whereas $\Irk{V}(k)$ for $k > n$ can be chosen freely to satisfy \eqref{V_constraint}.
That \eqref{V_constraint} is a relaxation of \eqref{Q_constraint} can be seen by noting that the two constraints become the same if $\Irk{V}(k)$ is further restricted to have the above form for all $k >n$. 

Next, temporarily assume that $\Irk{V}(k) = 0$ for $k > m$, for some $m>n$.
In this case $\Tfz{V}$ has the realization
\begin{align*}
\Tfz{V} = \left[ \begin{array}{c|c} \Phi_{V,m} & \Gamma_{V,m} \\ \hline \Pi_{V,m} & \Delta_V \end{array} \right]
\end{align*}
where
\begin{align*}
\Phi_{V,m} =& 
\begin{bmatrix} 0 & I_{2(m-1)n_u} \\ 0 & 0 \end{bmatrix}, &
\Gamma_{V,m} =& \begin{bmatrix} 0 \\ I_{2n_u} \end{bmatrix}, \\
\Pi_{V,m} =& \begin{bmatrix} \Irk{V}(m) & \cdots & \Irk{V}(1) \end{bmatrix}, &
\Delta_V =&  \Irk{V}(0)
\end{align*}
In this case, the KYP lemma gives that $\Tfz{V}$ satisfies \eqref{V_constraint} if and only if there exists a matrix $\Sigma_V = \Sigma_V^H$ such that 
\begin{equation}
\label{49495965uyhjhig9gughjg9gug}
\begin{bmatrix}
\Phi_{V,m}^H \Sigma_V \Phi_{V,m} - \Sigma_V & \Phi_{V,m}^H \Sigma_V \Gamma_{V,m} - \Pi_{V,m}^H \\
\Gamma_{V,m}^H \Sigma_V \Phi_{V,m} - \Pi_{V,m} & \Gamma_{V,m}^H \Sigma_V \Gamma_{V,m} - \Delta_V -\Delta_V^H \end{bmatrix} \preceq 0.
\end{equation}
In this context the minimization of \eqref{4885ghjgig9guhif8gvugvhjvigfiu} subject to \eqref{49495965uyhjhig9gughjg9gug} is equivalently stated as the minimax problem
\begin{equation*}
\inf\limits_{ \Sigma_V=\Sigma_V^H, \ \Irk{Q}(k), k \in \Set{Z}_{\geqslant 0} } \  \sup\limits_{\Xi=\Xi^H\succeq 0}
\check{\gamma}_{pd}
\end{equation*}
where $\check{\gamma}_{pd}$ is the primal-dual Lagrangian, i.e., 
\begin{align*}
&\check{\gamma}_{pd}
= \check{\gamma}_n
\nonumber \\ &  + \tr\left( \Xi
\begin{bmatrix}
\Phi_{V,m}^H \Sigma_V \Phi_{V,m} - \Sigma_V & \Phi_{V,m}^H \Sigma_V \Gamma_{V,m} - \Pi_V^H \\
\star & \Gamma_{V,m}^H \Sigma_V \Gamma_V - \Delta_V -\Delta_V^H \end{bmatrix}
\right)
\end{align*}
and $\Xi$ is a matrix of Lagrange multipliers. 
Because the problem is convex and Slater's condition holds, we can switch the minimization and maximization operations, and find the dual function as
\begin{align*}
\check{\gamma}_{d}
=
& \inf\limits_{ \Sigma_V=\Sigma_V^H, \ \Irk{Q}(k), k \in \Set{Z}_{\geqslant 0} } 
\Bigg\{ \check{\gamma}_n
- \tr\left( \Xi \begin{bmatrix} 0 & \Pi_{V,m}^H \\ \Pi_{V,m} & \Delta_V+\Delta_V^H \end{bmatrix} \right)
\nonumber \\ & 
+ \tr\left( \Sigma_{V,m} \left( \begin{bmatrix}\Phi_{V,m}^H \\ \Gamma_{V,m}^H\end{bmatrix}^H \Xi \begin{bmatrix}\Phi_{V,m}^H \\ \Gamma_{V,m}^H\end{bmatrix} - \begin{bmatrix} I &0 \\ 0 & 0 \end{bmatrix} \Xi \right) \right)
\Bigg\}
\end{align*}
In order for the above argument to be extremal in $\Sigma_V$, the term by which it multiplies in the trace above must be zero.  
This is the case if and only if $\Xi$ is a block Toeplitz matrix with each block being dimension $2n_u\times 2n_u$.

Defining
\begin{equation*}
\check{x}_n = \textrm{vec} \left\{ \Irk{Q}(0),...,\Irk{Q}(n) \right\}
\end{equation*}
it follows that there exist $\check{\gamma}_{n,0} \in \Set{R}_{>0}$, $\check{g}_n \in \Set{C}^{(n+1)n_u^2}$, and $\check{H}_n \in \Set{C}^{(n+1)n_u^2\times (n+1)n_u^2}$ with $\check{H}_n = \check{H}_n^H \succeq 0$ such that
\begin{equation} \label{56969707985994ut9g8gu}
\check{\gamma}_n = \check{\gamma}_{n,0} + \tfrac{1}{2} \left( \check{g}_n^H \check{x}_n + \check{x}_n^H \check{g}_n + \check{x}_n^H \check{H}_n \check{x}_n \right)
\end{equation}
where $\check{\gamma}_{n,0}$, $\check{g}_n$, and $\check{H}_n$ are quadratic functions of the first $(n+1)$ Markov coefficients of $\Tfz{P}_{fz}$, $\Tfz{P}_{uz}$, and $\Tfz{P}_{fy}$. 
For ease of notation, partition $\Xi$ as
\begin{align*}
\Xi = \begin{bmatrix} \Xi_{11} & \Xi_{12} & \Xi_{13} \\ \Xi_{12}^H & \Xi_{22} &\Xi_{23} \\ \Xi_{13}^H & \Xi_{23}^H & \Xi_{33} \end{bmatrix}
\end{align*}
where $\Xi_{22}$ has dimension $2nn_u\times 2nn_u$ and $\Xi_{33}$ has dimension $2n_u \times 2n_u$.
Also define $\check{\psi}_n$ as  
\begin{equation}
\check{\psi}_n \triangleq \textrm{vec} \left( \begin{bmatrix} \Xi_{23}^H & \Xi_{33} \end{bmatrix} \right).
\label{549tt9tjgh9h9gjtg09gjgh-9gj}
\end{equation}
Then there exists a matrix $\check{E}_n$ and a vector $\check{h}_n$ which depend only on the problem data, and such that 
\begin{align*}
-\tr\left( \Xi \begin{bmatrix} 0 & 0 \\ \Pi_{V,m} & \Delta_V \end{bmatrix} \right) 
=& \check{\psi}_n^H \left( \check{E}_n \check{x}_n + \check{h}_n \right)
\nonumber \\ &  \hspace{-12pt} -
 \tr\left( \begin{bmatrix} \Irk{V}(m) & \cdots & \Irk{V}(n+1) \end{bmatrix} \Xi_{13} \right).
\end{align*}
To be extremal we require that 
\begin{equation*}
\check{x}_n = -\check{H}_n^{-1} \left( \check{g}_n + 2\check{E}_n^H \check{\psi}_n \right)
\end{equation*}
where we have assumed the Hessian matrix $\check{H}_n$ is invertible.
(The case in which it is not can be handled similarly to the analysis in Section \ref{duality_section}.) 
We also require that $\Xi_{13} = 0$.
Note that because $\Xi$ is Toeplitz, this restricts its band to $n$, irrespective of how large $m$ is made. 

As such, the dual function is
\begin{align}
\check{\gamma}_{d,n} 
=& \check{\gamma}_{n,0} + \check{\psi}_n^H \check{h}_n + \check{h}_n^H \check{\psi}_n 
\nonumber \\ & 
- \tfrac{1}{2} \left( 2\check{E}_n^H \check{\psi}_n + \check{g}_n \right)^H \check{H}_n^{-1} \left( 2\check{E}_n^H \check{\psi}_n + \check{g}_n \right)
\label{595ggjg9g9hgjgh0gh9hj}
\end{align}
which is independent of $m$.
Taking the limit as $m\rightarrow\infty$, another application of the KYP Lemma gives that $\Xi \succeq 0$ if and only if there exists a matrix $\Sigma_\Xi \in \mathbb{C}^{2nn_u\times 2nn_u}$, with $\Sigma_\Xi = \Sigma_\Xi^H$, such that 
\begin{equation}
\begin{bmatrix} \Phi_{V,n}^H \Sigma_\Xi \Phi_{V,n} - \Sigma_\Xi & \Phi_{V,n}^H \Sigma_\Xi \Gamma_{V,n} - \Xi_{23} \\ \Gamma_{V,n}^H \Sigma_\Xi \Phi_{V,n} - \Xi_{23}^H & \Gamma_{V,n}^H \Sigma_\Xi \Gamma_{V,n} - \Xi_{33} \end{bmatrix}
\preceq 0 
\label{44845959574fgjgfurhygt}
\end{equation}

As such, we arrive at a finite-dimensional convex optimization to solve for a lower bound on optimality:
\begin{equation*}
\textrm{OP6}: \left\{ \begin{array}{ll}
\text{Given:} & \System{P}\in\Set{P}, \epsilon,\tau \in \Set{R}_{>0}, n \in \Set{Z}_{>0} \\
\text{Maximize:} & \check{\gamma}_{d,n} \text{ as in } \eqref{595ggjg9g9hgjgh0gh9hj} \\
\text{Domain:} & \check{\psi} \text{ as in } \eqref{549tt9tjgh9h9gjtg09gjgh-9gj}, \Sigma_\Xi = \Sigma_\Xi^H \\
\text{Constraint:} & \eqref{44845959574fgjgfurhygt}
\end{array} \right.
\end{equation*}
Denote as $\check{\gamma}_n^\star$ the optimal objective achieved in OP6.  
Then the following theorem establishes its properties relative to $\gamma^\star$, the optimized objective of OP3.

\begin{theorem} \label{lb_converence_theorem}
For each $n \in \Set{Z}_{>0}$, $\check{\gamma}_n^\star \leqslant \gamma^\star$.
Furthermore, $\check{\gamma}_n^\star$ is monotonically nondecreasing in $n$, with $\lim_{n\rightarrow\infty} \check{\gamma}_n^\star = \gamma^\star$.
\end{theorem}

\begin{rem} \label{dual_bound_comment}
The optimization domain of OP6 grows like $\mathcal{O}(n^2)$ due to the Lyapunov variable $\Sigma_\Xi$. As with OP4, this can be remedied by further usage of duality to obtain an optimization  with a domain that grows like $\mathcal{O}(n)$.  The process for this is similar to the conversion of OP4 to OP5.
\end{rem}

\section{Examples} \label{example}

\subsection{Example 1}

Suppose $\System{P}$ has a continuous-time realization \eqref{plant-input-output-model}, with the following model data
\begin{align*}
A =& \begin{bmatrix} 0 & 1 & 0 \\ -1 & -0.005 & -0.5 \\ 0 & 0.5 & -0.001 \end{bmatrix},
\ \ 
B_w = \begin{bmatrix} 0 \\ 1 \\ 0 \end{bmatrix}, 
\ \ 
B_u = \begin{bmatrix} 0 \\ 0 \\ 1 \end{bmatrix} \\
C_z =& \begin{bmatrix} -1 & -0.005 & -0.5 \\ 0 & 0 & 0 \end{bmatrix}, \quad
D_{uz} = \begin{bmatrix} 0 \\ 1 \end{bmatrix} \\
C_y =& \begin{bmatrix} 0 & 0 & 1 \end{bmatrix}, \quad 
D_{uy} = 0.
\end{align*}
The motivation for the above model is that it approximately characterizes a single degree-of-freedom vibratory oscillator coupled to a piezoelectric transducer, with the states corresponding to mechanical position, mechanical velocity, and capacitor voltage.
Control input $u$ is the current into the piezoelectric transducer, and disturbance input $w$ is the mechanical force on the oscillator.
The components of performance output vector $z$ contains the mechanical acceleration, the actuation current multiplied by $0.1$, and the mechanical velocity multiplied by $0.01$.
The model is nondimensionalized with natural frequency of $1$, fraction of critical damping of $0.25\%$, piezoelectric coupling factor of $0.5$, and a piezoelectric leakage time constant equal to $1000$.
We note that the unconstrained $\Tfs{\Set{H}}_2$-optimal controller is
not PR. (Indeed, it is open-loop unstable.)

\begin{figure}
\centering
\includegraphics[scale=.7]{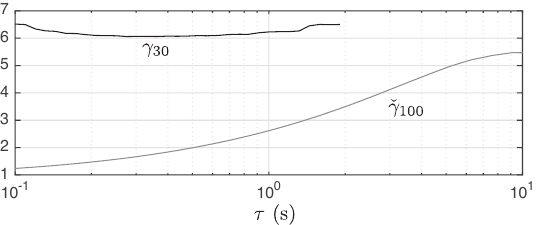}
\caption{Solutions to OP5 with $n=30$ (i.e., $\gamma_{30}$) and OP6 with $n=100$ (i.e., $\check{\gamma}_{100}$) over a domain of $\tau$ values.  Outside the support domain of $\gamma_{30}$, numerical problems prevented accurate solutions from being obtained for OP5.}
\label{example1_tau}
\end{figure}

\begin{figure}
\centering
\includegraphics[scale=.7]{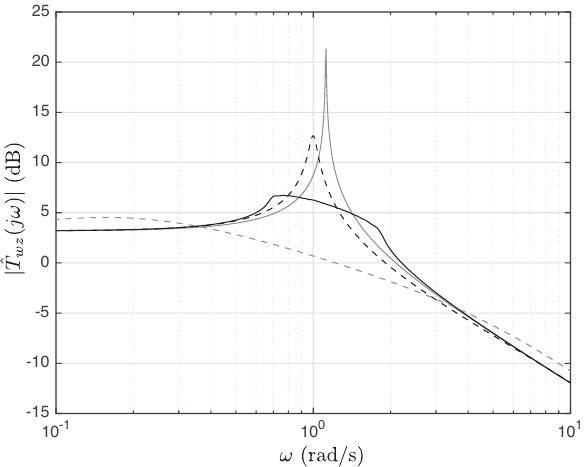}
\caption{$|\hat{T}_{wz}(j\omega)|$ for the uncontrolled system (gray), as well as closed-loop with controllers $\Tfs{K}_0$ (dash, black) and $\Tfs{K}_{100}^\star$ (solid, black), and unconstrained optimum (dash, gray) for Example 1}
\label{example1_bode}
\end{figure}

For this plant, we presume $\epsilon = 0.1$.
To choose a value of $\tau$, we note that for this example suitable values of $\tau$ are different for optimizing $\gamma_n$ versus $\check{\gamma}_n$.
Both OP5 and OP6 (especially if the dual formulation discussed in Remark \ref{dual_bound_comment} is implemented for OP6) are highly efficient to evaluate, even for $n$ values in excess of $100$.
\footnote{Solution times with $n = 100$ are only a few seconds, implemented in Matlab using the CVX parser with the Sedumi algorithm, and running on a 2021 MacBook Pro with an Apple M1 chip and 16GB of memory.}
As such, for a given value of $n$ it requires minimal computational effort to search for the best $\tau$ values by nesting OP5 and OP6 inside a line search.
For example, Figure \ref{example1_tau} shows values of $\gamma_n$ for $n=30$, evaluated over a domain of $\tau$. 
These results suggest $\tau = 0.5\text{s}$ is reasonable for OP5 with $n=30$.
Similarly, the plot shows $\check{\gamma}_n$ for $n=100$, evaluated over for the same domain in $\tau$.
These results indicate that $\tau = 10\text{s}$ is reasonable for OP6 for $n=100$. 
Henceforth we use these respective values of $\tau$, as $n$ is varied.

\begin{figure}
\centering
\includegraphics[scale=.7]{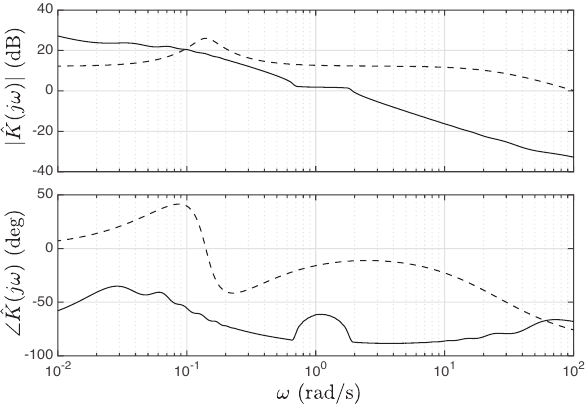}
\caption{$\Tfs{K}_0$ (dash, black) and $\Tfs{K}_{100}^\star$ (solid, black) for Example 1}
\label{example1_controller_bode}
\end{figure}

\begin{figure}
\centering
\includegraphics[scale=.7]{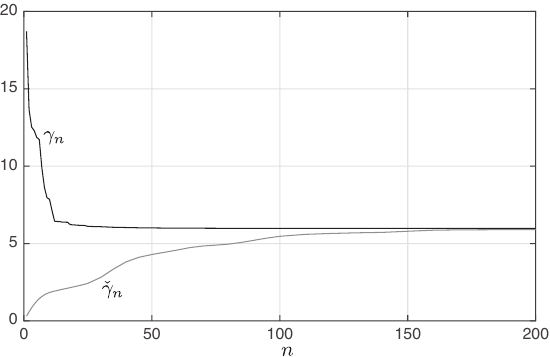}
\caption{Optimized objective $\gamma_n^\star$ and lower bound $\check{\gamma}_n^\star$ normalized by the value of the unconstrained $\Tfs{\Set{H}}_2$-optimal controller for Example 1}
\label{example1_performance_attained}
\end{figure}

Figure \ref{example1_bode} shows the magnitude of mapping $w \mapsto z$ for four scenarios: uncontrolled, closed-loop with controller $\Tfs{K}_0$, closed-loop with controller $\Tfs{K}_{100}^\star$, and unconstrained optimal.
Figure \ref{example1_controller_bode} shows the corresponding sub-optimal and optimal control designs.
The controller is positive-real, as expected, and there is a rather marked difference between $\Tfs{K}_0$ and $\Tfs{K}_{100}^\star$.
In particular, we note the unusual ``notch'' in the phase of $\Tfs{K}_{100}^\star(j\omega)$ for $\omega \in [0.5,1.1]$, over which the magnitude of nearly constant. 
Figure \ref{example1_performance_attained} shows the objective $\gamma_n^\star$ (normalized by the unconstrained optimal objective) attained by the optimal controller as a function of $n$, as well as the lower bound $\check{\gamma}_n^\star$.
This plot indicates that with $n = 30$, it is possible to reach near-optimality for the truncated solution.
However, the lower bound converges more slowly.
Nonetheless this result corroborates the claim that both $\gamma_n^\star$ and $\check{\gamma}_n^\star$ converge to $\gamma^\star$. 

\subsection{Example 2}
For this example, plant $\System{P}$ is characterized by
\begin{align*}
A =& \begin{bmatrix} 0_{4\times 4} & I_4 & 0_{4\times 1} \\ -M_P^{-1} K_P & -M_P^{-1} C_P & G_{P2} \\
0_{1\times 4} & 0_{1\times 4} & -10 \end{bmatrix}, \\
B_w =& \begin{bmatrix} 0_{4\times 1} \\ G_{P1}-G_{P2} \\ 20 \end{bmatrix}, \quad 
B_u = \begin{bmatrix} 0_{4\times 2} \\ M_P^{-1} N_P \\ 0_{1\times 2} \end{bmatrix}, \\
C_z =& \begin{bmatrix} \hat{e}_1 & \hat{e}_2 \end{bmatrix}^T \begin{bmatrix} -M_P^{-1}K_P & -M_P^{-1} C_P & 0_{4\times 1} \end{bmatrix} \\
D_{uz} =& \begin{bmatrix} \hat{e}_1 & \hat{e}_2 \end{bmatrix}^T M_P^{-1} N_P \\
C_y =& \begin{bmatrix} 0_{2\times 4} & N_P^T & 0_{2\times 1} \end{bmatrix}, \quad D_{uy} = 0_{2\times 2}
\end{align*}
where $\hat{e}_i$ is the Cartesian unit vector in direction $i$, and 
\begin{align*}
M_P =& \begin{bmatrix} 10 & 5 & 0 & 0 \\ 5 & 10 & 0 & 0 \\ 0 & 0 & 1 & 0 \\ 0 & 0 & 0 & 1 \end{bmatrix}
, & 
K_P =& \begin{bmatrix} 1 & 0 & -1 & 0 \\ 0 & 1 & 0 & -1 \\ -1 & 0 & 10 & 0 \\ 0 & -1 & 0 & 10 \end{bmatrix}
\end{align*}
\begin{align*}
C_P =& 0.01 \begin{bmatrix} 1 & 0 & -1 & 0 \\ 0 & 1 & 0 & -1 \\ -1 & 0 & 1 & 0 \\ 0 & -1 & 0 & 1 \end{bmatrix}, &
N_P =& \begin{bmatrix} 1 & 0 \\ 0 & 1 \\ -1 & 0 \\ 0 & -1 \end{bmatrix}
\end{align*}
\begin{align*}
G_{P1} =& \begin{bmatrix} 1 & 0 & 1 & 0 \end{bmatrix}^T, &
G_{P2} =& \begin{bmatrix} 0 & 1 & 0 & 1 \end{bmatrix}^T
\end{align*}
This system represents the linearized and nondimensionalized dynamics of a two-dimensional automotive suspension, in which the first four states are, respectively, the vertical displacements of the front and rear cabin, and the vertical displacements of the front and rear chassis, respectively.  The next four states are the corresponding velocities.
The disturbance $w$ is taken to be the second time-derivative of the road elevation at the front tires, modeled as white noise. The last state is used to approximate the delay between this disturbance, and the same corresponding disturbance injection at the rear tires. It is assumed that the only mechanical damping in the system is comprised of viscous dampers between the cabin and chassis. Control actuators are situated between the cabin and chassis as well, at both the front and rear axles. Performance outputs $z$ are the absolute accelerations of the cabin at the front and rear axles, and the control force inputs multiplied by $10^{-3}$. 

We presume  $\epsilon = 0.5$.
Conducting an line search analogous to the previous example to find suitable values of $\tau$, Figure \ref{example2_tau} shows values of $\gamma_n^\star$ and $\check{\gamma}_n^\star$, both for $n = 35$. In this example the same value of $\tau$ can be justifiably used for both, and here we choose $\tau = 0.9\text{s}$ for all subsequent evaluations.

\begin{figure}
\centering
\includegraphics[scale=.7]{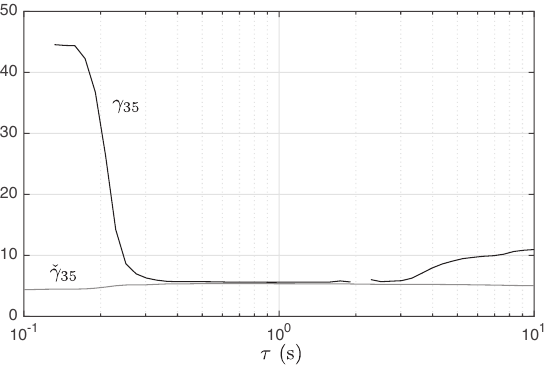}
\caption{Solutions to OP5 with $n=35$ (i.e., $\gamma_{35}$) and OP6 with $n=35$ (i.e., $\check{\gamma}_{35}$) over a domain of $\tau$ values.  Outside the support domain of $\gamma_{35}$, numerical problems prevented accurate solutions from being obtained for OP5.}
\label{example2_tau}
\end{figure}

\begin{figure}
\centering
\includegraphics[scale=.7]{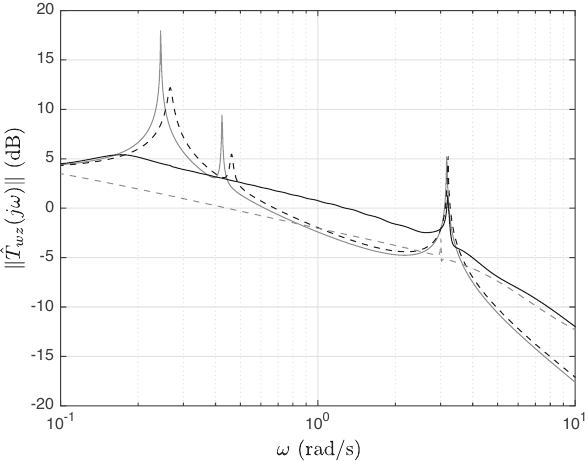}
\caption{$\|\hat{T}_{wz}(j\omega)\|$ for the uncontrolled system (solid gray), as well as closed-loop with controllers $\Tfs{K}_0$ (dash, black), $\Tfs{K}_{50}^\star$ (solid black), and unconstrained optimum (dash, gray) for Example 2}
\label{example2_bode}
\end{figure}

\begin{figure}
\centering
\includegraphics[scale=.7]{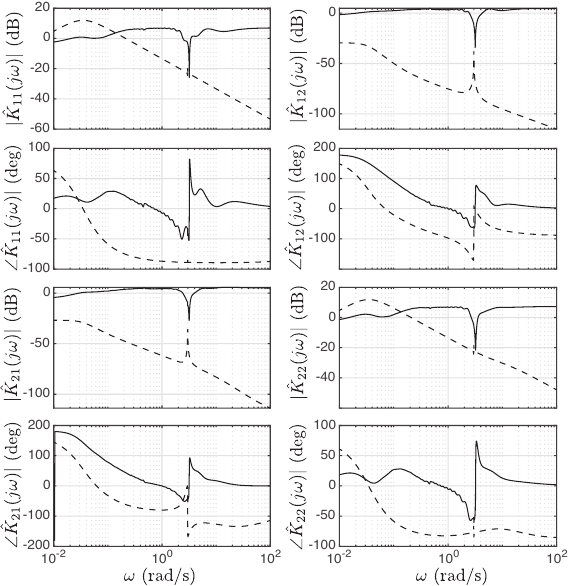}
\caption{$\Tfs{K}_0$ (dash, black) and $\Tfs{K}_{50}^\star$ (solid black) for Example 2}
\label{example2_controller_bode}
\end{figure}

\begin{figure}
\centering
\includegraphics[scale=.7]{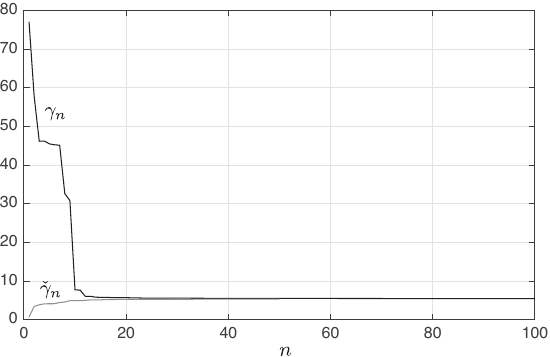}
\caption{Performance $\gamma_n^\star$ and $\check{\gamma}_n^\star$, normalized by the value of the unconstrained $\Tfs{\Set{H}}_2$-optimal controller for Example 2}
\label{example2_performance_attained}
\end{figure}

Figure \ref{example2_bode} shows the magnitude of response $\| \Tfs{T}_{wz}(j\omega) \|_2$ for the open-loop system, the closed-loop system with optimal controller $\Tfs{K}^\star_{50}$, and the closed-loop system with sub-optimal controller $\Tfs{K}_0$. 
Figure \ref{example2_controller_bode} is a Bode plot the $2\times 2$ matrix transfer function $\Tfs{K}_{50}^\star$ as well as $\Tfs{K}_0$. 
Finally, Figure \ref{example2_performance_attained} shows the convergence of objective $\gamma_n^\star$ and lower bound $\check{\gamma}_n^\star$ as $n$ is increased.
These results prompt many of the same general conclusions as in Example 1.

\section{Some Extensions} \label{extensions}

%\subsection{Generalization to other performance measures}

The scope of this paper was restricted to $\Set{H}_2$ performance objectives, but the methodology employed here applies more broadly, often with only minor modifications.

For example, extension to generalized-$\Set{H}_2$ performance objectives straightforward. 
Such objectives take the form
\begin{equation*}
\gamma = \max\limits_{i\in\{1,...,n_\gamma\}} \| \Tfs{T}_{w_iz_i} \|_{\Tfs{\Set{H}}_2}^2
\end{equation*}
where for each $i$, $w_i$ and $z_i$ are predefined subspaces of the overall input and output vectors respectively. 
Extension of the methodology to objectives of this form requires no substantive modifications the optimization methodology, with OP4 being modified only in the evaluation of $\gamma$, which becomes $\max\{\gamma_1,...,\gamma_{n_\gamma}\}$ where each $\gamma_i$ has the form of considered in this paper. 
Dual optimization OP5, as well as lower-bound optimization OP6, can also be formulated through standard use of Lagrange multipliers.
We do note that the sub-optimal control design methodology from Section \ref{suboptimal_solution} cannot be applied in the generalized $\Set{H}_2$ case unless $w_i$ is the same for all $i\in\{1,...,n_\gamma\}$ because otherwise certainty-equivalence does not hold. 

Extension to $\Set{H}_\infty$ performance objectives is also possible, with appropriate alterations to the methodology. 
One important change is in the relation of the performance measure to the discrete-time Youla parameter $\Tfz{Q}$, which for the case considered in the paper, is equation \eqref{Tfz_norm} with $\Tfz{T}_{fz}$ evaluated as in \eqref{Tfz_from_Q}.
In an $\Set{H}_\infty$ context, 
$
\gamma = \| \Tfz{T}_{wz} \|_{\Tfz{\Set{H}}_\infty}
$
with it being unnecessary to introduce the new discrete-time input $f$.
Another important change is that $\gamma$ must be evaluated using the KYP Lemma and an auxiliary Hermitian multiplier matrix, rather than being evaluated directly as a quadratic function of the control variables.
This results in some differences in the optimization domains of OP5 and OP6.

\section{Conclusions} \label{conclusions}

Feedback passivity affords certain advantages in many control applications. 
It allows for a feedback controller to be implemented in a manner which, theoretically, requires no power in order to operate perpetually.
Moreover, when the plant is itself passive, such controllers have unconditional guarantees on stability-robustness.
These observations provide a motivation to determine, with some precision, the best objective that is achievable with a passive feedback law.
This paper has presented a technique for accomplishing this, in the context of $\Set{H}_2$-optimal control.

The key contribution of this work is the recognition that feedback passivity may be imposed as a constraint on an optimal control problem,  in a manner that is convex.
This is done by explicit optimization of the Youla parameter associated with the feedback controller.
We have shown that in the context of $\Set{H}_2$-optimal control, this leads to an efficient and convex algorithm to solve for a controller which is optimal, given a prescribed number $n$ of nonzero Markov coefficients for the Youla parameter.
We have, furthermore, shown that the optimization domain for this algorithm grows linearly with $n$.

Although we have touched on a few extensions of these ideas to other types of objectives many questions remain.
The most important is the extension of the techniques considered here to guarantee bounds on robust performance, rather than just the assurance of robust stability.
To the author's knowledge, this remains an open problem.

Finally, we note the important role that control complexity often plays in problems involving feedback-passivity constraints. Our analysis here does not account for the favorability of controllers with low order, compared to those with higher order.  Additionally, the techniques proposed here do not allow for feedback laws to be decentralized (without sacrificing convexity). 
Extensions to accommodate such constraints would greatly enhance the domain of problems for which the proposed techniques have relevance.

\appendix

\subsection{Proof of Proposition \ref{485696879790786}}

We first note that $\gamma < \gamma_u + \gamma_m$ for some $\gamma_m\in\Set{R}$ if and only if there exist matrices $X=X^H$ and $Y=Y^H$ with $X \succ 0$ and $Y \succ 0$, such that
$\tr( C_K X C_K^H R ) <  \gamma_m$, and
\begin{align}
\label{5958568799707848499}
AX+ XA^H + B_w B_w^H 
+ Y C_K^H C_y X 
& \nonumber \\
+ X C_y^H C_KY & \preceq 0
\\
\label{585784948373939578595}
A_Z Y + Y A_Z^H +Y [  C_K^HC_Z + C_Z^H C_K & \nonumber \\ + \epsilon C_K^H C_K ] Y & \preceq 0.
\end{align}

Inequality \eqref{585784948373939578595} is nonconvex due to the fact that the term in the brackets is not positive semidefinite. 
However, this may be conservatively overbounded by recognizing that
\begin{align*}
& Y [  C_K^HC_Z + C_Z^H C_K + \epsilon C_K^H C_K ] Y 
\nonumber \\ 
&=  Y G_p^H G_p Y - Y G_n^H G_n Y \\
&\preceq Y G_p^H G_p Y + Z_0^{-1} G_n^H G_n Z_0^{-1}  
\nonumber \\ & \quad 
- Y G_n^H G_n Z_0^{-1} - Z_0^{-1} G_n^H G_n Y
\end{align*}
As such, \eqref{4949377384950960687659} implies \eqref{585784948373939578595} via a Schur complement.

Inequality \eqref{5958568799707848499} is nonconvex due to bilinear terms in variables $X$ and $Y$. It may be conservatively overbounded by a convex inequality by recognizing that for any $\Psi=\Psi^H\succ 0$, 
\begin{align*}
& [Y-Z_0^{-1}] C_K^HC_y [X-X_0] + [X-X_0] C_y^HC_K [Y-Z_0^{-1}]
\nonumber \\ &
\preceq [Y-Z_0^{-1}] C_K^H \Psi C_K [Y-Z_0^{-1}] 
\nonumber \\ & \quad
+ [X-X_0] C_y^H \Psi^{-1} C_y [X-X_0] 
\end{align*}
As such, subtracting the Lyapunov equation for $X_0$,  \eqref{5958568799707848499} is conservatively satisfied by 
\begin{align*}
A[X-X_0]+ [X-X_0]A^H &
\nonumber \\
+ [Y-Z_0^{-1}] C_K^H C_y X_0 
+ X_0 C_y^H C_K [Y-Z_0^{-1}]  &
\nonumber \\
+ Z_0^{-1} C_K^H C_y [X-X_0]
+ [X-X_0] C_y^H C_K Z_0^{-1} &
\nonumber \\
+ [Y-Z_0^{-1}] C_K^H \Psi C_K [Y-Z_0^{-1}]  &
\nonumber \\ 
+ [X-X_0] C_y^H \Psi^{-1} C_y [X-X_0]
& \preceq 0
\end{align*}
which is equivalent to constraint \eqref{0303838576869687484859} via Schur complements.

Next we show that $A_C$ is Hurwitz, using the Passivity and Separation Theorems.
For $B_K = Z_0^{-1} C_K^H$, we have that
\begin{align*}
& [A_Z+B_K(C_y+D_{uy}C_K)]^H Z_0 \nonumber \\ & + Z_0 [A_Z+B_K(C_y+D_{uy}C_K)] + \epsilon C_K^H C_K \preceq 0
\end{align*}
and $Z_0 \succ 0$, implying that $[A_Z+B_K(C_y+D_{uy}C_K)]$ has no positive-real eigenvalues, and any pure-imaginary eigenvalue $\lambda$ must have an associated eigenvector $\eta$ satisfying $C_K\eta = 0$. If this is the case then it implies that $\lambda$ is also an eigenvalue of $A+B_KC_y$ and $\eta$ the associated eigenvector.
Furthermore, because $B_K^H Z_0 \eta = 0$, it implies that
\begin{align*}
0 =& 
A_Z^H Z_0 \eta + Z_0 \eta \lambda 
\end{align*}
implying that $\lambda$ is an imaginary eigenvalue of $A_Z$ with associated eigenvector $Z_0 \eta$, which is nonzero because $Z_0 \succ 0$.  But $A_Z$ cannot have imaginary eigenvalues if the conditions of Theorem \ref{4949576070748303} hold, because it must be Hurwitz. 
As such we have a contradiction, and therefore conclude that $[A_Z+B_K(C_y+D_{uy}C_K)]$ is Hurwitz. 
Next, consider that both the systems
\begin{equation*}
\left[ \begin{array}{c|c} A & B_u \\ \hline C_y & D_{uy} \end{array} \right], \ 
\left[ \begin{array}{c|c} A+B_K(C_y+D_{uy}C_K)+B_uC_K & B_K \\ \hline C_K & 0 \end{array} \right]
\end{equation*}
are passive, and that both $A$ and $A+B_K(C_y+D_{uy}C_K)+B_uC_K$ are Hurwitz.
Consequently, via the Passivity Theorem, their negative feedback interconnection is asymptotically stable, implying that the matrix
\begin{equation*}
\begin{bmatrix} A & B_u C_K \\ -B_K C_y & A+B_K C_y + B_u C_K\end{bmatrix}
\end{equation*}
is Hurwitz. But this matrix is similar to 
\begin{equation*}
\begin{bmatrix} A_C & 0 \\ -B_K C_y & A_Z \end{bmatrix}
\end{equation*}
so we conclude that both $A_C$ is Hurwitz. 

It remains to show that feasibility is guaranteed.
Because $A_C$ is Hurwitz it is immediate that all solutions $X_0$ to inequality \eqref{4948567686878797} must be positive semidefinite, and that there exists $X_0 \succ 0$ satisfying the inequality. 
Noting that $Y = Z_0^{-1}$ satisfies \eqref{4949377384950960687659}, we have that $Y = Z_0^{-1}$ and $X = X_0$ are feasible.

\subsection{Proof of Theorem \ref{59595686877005}}

($\Rightarrow$) If $\System{K} \in \Set{K}_\epsilon$ then $\Tfs{K} \in \Tfs{\Set{H}}_\infty$ and \eqref{K_constraint} holds for all $s \in \Set{C}_{>0}$.
It follows that $\Tfz{K} \in \Tfz{\Set{H}}_\infty$ and
\begin{equation}
\Tfz{K}^H(q) + \Tfz{K}(q) - \epsilon \Tfz{K}(q) \Tfz{K}^H(q) \succ 0, \ \forall q \in \Set{O}_{>1}.
\label{proof_necessity_6574847567}
\end{equation} 
Substitute \eqref{K_from_Q}, pre-multiply by $\eta^H [I-\Tfz{Q}(q)\Tfz{P}_{uy}(q)]$ and post-multiply by $[I-\Tfz{Q}(q)\Tfz{P}_{uy}(q)]^H\eta$ for an arbitrary vector $\eta$ which does not lie in the null space of $[I - \Tfz{Q}(q) \Tfz{P}_{uy}(q)]^H$.
We then have that 
\begin{multline}
\eta^H \Big\{ \Tfz{Q}^H(q) + \Tfz{Q}(q) \\
- \Tfz{Q}(q) \big[ \Tfz{P}_{uy}^H(q)+\Tfz{P}_{uy}(q)+\epsilon I \big] \Tfz{Q}^H(q) \Big\} \eta > 0
\label{necessity_proof_47585757}
\end{multline}
Meanwhile if $[I - \Tfz{Q}(q) \Tfz{P}_{uy}(q)]^H \eta = 0$ then, from \eqref{K_from_Q}, it follows that 
\begin{equation}
\eta^H \left[ I - \Tfz{Q}(q) \Tfz{P}_{uy}(q) \right] \Tfz{K}(q) = \eta^H \Tfz{Q}(q)  
\end{equation}
Because $\| \Tfz{K}(q) \| \leqslant \tfrac{2}{\epsilon}$ it follows that the left-hand side is zero, and consequently that $\eta^H \Tfz{Q}(q) = 0$. 
If this is the case then \eqref{necessity_proof_47585757} holds.
We conclude that 
\begin{multline*}
\Tfz{Q}^H(q) + \Tfz{Q}(q) \\
- \Tfz{Q}(q) \big[ \Tfz{P}_{uy}^H(q)+\Tfz{P}_{uy}(q)+\epsilon I \big] \Tfz{Q}^H(q) \succeq 0
\end{multline*}
for all $q \in \Set{O}_{>1}$. 
If this is true, it also holds on the boundary of $\Set{O}_{>1}$, giving \eqref{Q_constraint} through a Schur complement.
To show that $\Tfz{Q}$ is analytic on $\Set{O}_{>1}$, we note that $\Tfz{K}$ and $\Tfz{P}_{uy}$ are, and thus it follows from \eqref{Q_from_K} that $\Tfz{Q}$ is, assuming it is uniformly bounded.
To show $\Tfz{Q}$ uniformly bounded in $\Set{O}_{>1}$, note that because \eqref{posreal_discrete_time_1} is true, the above conservatively implies that
\begin{equation*}
\Tfz{Q}^H(q) + \Tfz{Q}(q) - \epsilon \Tfz{Q}(q) \Tfz{Q}^H(q) \succeq 0, \ \forall q \in \Set{O}_{>1}.
\end{equation*}
If this is the case then $\| \Tfz{Q}(q) \| \leqslant \tfrac{2}{\epsilon}$ for all $q \in \Set{O}_{>1}$.
 
($\Leftarrow$) Assuming \eqref{Q_constraint}, let $\psi \in \Set{R}_{>0}$ be such that 
\begin{equation}
\sup\limits_{q\in\Set{O}_{>1}} \| \psi \Tfz{U}(q) \| < 1
\label{sufficiency_proof_56848467r487}
\end{equation}  
(Note that such a $\psi$ is guaranteed to exist because $\Tfz{Q}$ and $\Tfz{F}$ are both in $\Tfz{\Set{H}}_\infty$.) 
Then \eqref{Q_constraint} implies that 
\begin{equation*}
\left\| \left[I-\psi\Tfz{U}(e^{j\Omega}) \right] \left[ I + \psi\Tfz{U}(e^{j\Omega}) \right]^{-1} \right\| \leqslant 1, \ \forall \Omega \in [-\pi,\pi].
\end{equation*} 
Because \eqref{sufficiency_proof_56848467r487} holds, it follows that $[I - \psi\Tfz{U}] [I+\psi\Tfz{U}]^{-1}$ is analytic on $\Set{O}_{>1}$.
It consequently follows from the Maximum Modulus Theorem that 
\begin{equation*}
\left\| \left[I-\psi\Tfz{U}(q) \right] \left[ I + \psi\Tfz{U}(q) \right]^{-1} \right\| \leqslant 1, \ \forall q \in \Set{O}_{>1}.
\end{equation*} 
and therefore that 
\begin{equation*}
\Tfz{U}^H(q) + \Tfz{U}(q) \succeq 0, \ \forall q \in \Set{O}_{>1}.
\end{equation*}
Taking a Schur complement and rearranging, we have that
\begin{multline*}
\Tfz{Q}(q) [I - \Tfz{P}_{uy}^H(q) \Tfz{Q}^H(q)] + [I - \Tfz{Q}(q)\Tfz{P}_{uy}(q) ] \Tfz{Q}^H(q) 
\\ - \epsilon \Tfz{Q}(q) \Tfz{Q}^H(q) \succeq 0
\end{multline*}
Now suppose $\xi \neq 0$ is in the null space of $[I - \Tfz{Q}(q) \Tfz{P}_{uy}(q)]^H$.
Then it follows that $-\epsilon \xi^H \Tfz{Q}(q) \Tfz{Q}^H(q) \xi \geqslant 0$, implying that $\Tfz{Q}^H(q) \xi = 0$. But $\Tfz{P}_{uy}(q)$ is uniformly bounded for all $q \in \Set{O}_{>1}$ so we conclude that $\Tfz{P}_{uy}^H(q) \Tfz{Q}^H(q) \xi = 0$ and therefore $[I - \Tfz{Q}(q) \Tfz{P}_{uy}(q)]^H \xi = \xi \neq 0$, a contradiction.  We therefore conclude that $[I - \Tfz{Q}(q) \Tfz{P}_{uy}(q)]$ is full-rank for all $q \in \Set{O}_{>1}$. 
Post-multiply the above equation by $[I - \Tfz{Q}(q) \Tfz{P}_{uy}(q)]^{-H}$ and pre-multiply by $[I - \Tfz{Q}(q) \Tfz{P}_{uy}(q)]^{-1}$ to get \eqref{proof_necessity_6574847567}.
It remains to show that $\Tfz{K} \in \Tfz{\Set{H}}_\infty$.
To show that it is analytic on $\Set{O}_{>1}$, we first note that because $[I-\Tfz{Q}(q)\Tfz{P}_{uy}(q)]$ is nonsingular for all $q \in \Set{O}_{>1}$ and $\Tfz{Q}(q)$ and $\Tfz{P}_{uy}(q)$ are analytic, it follows that $[I-\Tfz{Q}(q)\Tfz{P}_{uy}(q)]^{-1}$ is analytic as well.
Using \eqref{K_from_Q}, we conclude that $\Tfz{K}(q)$ is analytic on $\Set{O}_{>1}$.
To show that it is uniformly bounded, we note that because the equation above holds, it follows that $\| \Tfz{K}(q) \| \leqslant \tfrac{2}{\epsilon}$ for all $q \in \Set{O}_{>1}$. 

\subsection{Proof of Theorem \ref{uniqueness_theorem}}

Suppose $\Tfz{Q}^\star_a$ and $\Tfz{Q}^\star_b$ are two minimizers, each rendering $\gamma = \gamma^\star$. 
Because OP3 is convex, it follows that all interpolated points $\Tfz{Q}^\star_\alpha = \Tfz{Q}^\star_a + \alpha(\Tfz{Q}^\star_b-\Tfz{Q}^\star_a)$ for $\alpha \in [0,1]$ are also minimizers achieving the same $\gamma = \gamma^\star$. 
Let $\Tfz{T}_{fz\alpha}^\star$ be the closed-loop transfer function corresponding to Youla parameter $\Tfz{Q}^\star_\alpha$.
Then 
\begin{align*}
0 =& \frac{1}{2} \frac{\partial^2}{\partial\alpha^2} 
\| \Tfz{T}_{fz\alpha} \|_{\Tfz{\Set{H}}_2}^2
\\
=& \frac{1}{2\pi} \int_{-\pi}^\pi \| 
\Tfz{P}_{uz}(e^{j\Omega})
[\Tfz{Q}_b^\star(e^{j\Omega})-\Tfz{Q}_a^\star(e^{j\Omega}] \Tfz{P}_{fy}(e^{j\Omega}) 
\|_2^2 d\Omega
\end{align*}
Now recall that $\Tfz{P}_{fy}$ and $\Tfz{P}_{uz}$ are uniformly bounded in $\Set{O}_{>1}$, and also that  constraint \eqref{Q_constraint} requires $\sup_{\Omega\in[-\pi,\pi]} \| \Tfz{Q}_a^\star(e^{j\Omega}) \| \leqslant \tfrac{2}{\epsilon}$, and similarly for $\Tfz{Q}_b^\star$. 
As such, each multiplicative component in the above integral is individually bounded.
Furthermore, the integrand itself is positive-semidefinite in $\Tfz{Q}_b^\star(e^{j\Omega}) - \Tfz{Q}_a^\star(e^{j\Omega})$, and is zero if and only if $\Omega \in \Set{W}'$.  
So if $\Set{W}'$ has zero measure, then 
\begin{align*}
&\int_{-\pi}^\pi \| \Tfz{Q}_b^\star(e^{j\Omega})-\Tfz{Q}_a^\star(e^{j\Omega}) \|_2^2 d\Omega  = 0
\end{align*}
implying that $\Tfz{Q}_b^\star = \Tfz{Q}_a^\star$ except on a set of zero measure.
Via the Plancharel theorem, this implies $\Irk{Q}_{a}^\star(k) = \Irk{Q}_{b}^\star(k)$,  $\forall k\in\Set{Z}_{\geqslant 0}$.
Finally, \eqref{49485768jgig8uhgh} follows from $\sup_{\Omega\in[-\pi,\pi]} \| \Tfz{Q}^\star(e^{j\Omega}) \| \leqslant \tfrac{2}{\epsilon}$, which implies it through Plancharel's theorem.

\subsection{Proof of Theorem \ref{45050607007594848586}}

Define $\Tfz{Q}_a^\star$ and $\Tfz{Q}_b^\star$ be solutions to OP3 with bilinear parameters $\tau_a$ and $\tau_b$ respectively. 
Then the solution for $\tau_a$ is associated with the continuous-time transfer function
\begin{equation*}
\Tfs{Q}_a^\star(s) \triangleq \Tfz{Q}_{a}^\star \left( \frac{1+2\tau_a s}{1-2\tau_a s} \right)
\end{equation*}
Then define the $\Tfz{Q}_{ab}(q)$ as the bilinear transformation of $\Tfs{Q}_{a}^\star$ with parameter $\tau_b$, i.e., 
\begin{align*}
\Tfz{Q}_{ab}(q) 
\triangleq & \Tfs{Q}_a^\star \left( \frac{1}{2\tau_b} \frac{q-1}{q+1} \right) 
= \Tfz{Q}_a^\star \left( \frac{1 +  \frac{\tau_a}{\tau_b} \frac{q-1}{q+1}}{1- \frac{\tau_a}{\tau_b} \frac{q-1}{q+1} } \right) 
\end{align*}
which is analytic and uniformly bounded for all $q \in \Set{O}_{>1}$ if $\Tfz{Q}_a^\star(q)$ is.
Furthermore, noting that $\Tfz{Q}_{ab}(e^{j\Omega_b}) = \Tfz{Q}_a^\star(e^{j\Omega_a})$ where
\begin{equation*}
\Omega_b = 2 \tan^{-1} \left[ \frac{\tau_b}{\tau_a} \tan(\tfrac{1}{2}\Omega_a) \right]
\end{equation*}
we conclude that $\Tfz{Q}_{ab}$ is feasible if and only if $\Tfz{Q}_a^\star$ is. 
Now suppose $\gamma_a^\star < \gamma_b^\star$.
Then this results in a contradiction because the value of $\gamma$ obtained with $\Tfz{Q}_{ab}$ is lower than $\gamma_b^\star$, and is feasible. 
As such, $\gamma_a^\star = \gamma_b^\star$.

To prove that $\| \Tfs{Q}_a^\star - \Tfs{Q}_b^\star \|_{\Tfs{\Set{H}}_2} = 0$, assume it is not the case. 
Then from the reasoning above, this implies that there exists a $\Tfz{Q}_{ab}$ such that $\Tfz{Q}_{ab}(e^{j\Omega}) \neq \Tfz{Q}_{b}^\star(e^{j\Omega})$ for all $\Omega$ in a set of finite measure, and which also attains the optimal objective.
This implies that $\Tfz{Q}_b^\star$ is nonunique.

\subsection{Proof of Claim \ref{quadratic_performance_objective}}
First recognize that $\Tfz{T}_{fz}$ as in \eqref{373484758567} can be represented via the partitioned state space realization
\begin{equation}
\Tfz{T}_{fz} = \left[ \begin{array}{l|l} \Phi_T & \Gamma_T \\ \hline \Pi_T & 0 \end{array} \right]
=
\left[ \begin{array}{ll|l}
\Phi_{T11} & \Phi_{T12}(x_n) & \Gamma_{T1} \\
0 & \Phi_{T22} & \Gamma_{T2} \\ \hline
\Pi_{T1} & \Pi_{T2}(x_n) & 0 
\end{array} \right]
\label{49474768697978478474}
\end{equation}
$\Phi_{T12}(x_n)$ and $\Pi_{T2}(x_n)$ are affine in $x_n$, while all other parameters are independent of $x_n$.
We have that 
\begin{equation*}
\gamma_n(x_n) = \tfrac{1}{\tau} \tr\left\{ \Gamma_T^H \Sigma_T \Gamma_T \right\}
\end{equation*}
where $\Sigma_T=\Sigma_T^H \succ 0 $ is the solution to 
\begin{equation} \label{430475606}
0 = \Phi_T^H \Sigma_T \Phi_T - \Sigma_T + \Pi_T^H \Pi_T
\end{equation}
Let $\Sigma_T$ be factored as
\begin{align*}
\Sigma_T
=& \begin{bmatrix} \Sigma_{T1}^{-1} & 0 \\ -\Lambda_T^H & I \end{bmatrix}^{-1} \begin{bmatrix} I & \Lambda_T \\ 0 & Y_T \end{bmatrix}
\end{align*}
where $\Sigma_{T1}=\Sigma_{T1}^H\succ 0$ and $Y_T=Y_T^H \succ 0$. 
Then, noting that
\begin{align*}
& \begin{bmatrix} \Sigma_{T1}^{-1} & 0 \\ -\Lambda_T^H & I \end{bmatrix} \Sigma_T \Phi_T \begin{bmatrix} \Sigma_{T1}^{-1} & -\Lambda_T \\ 0 & I \end{bmatrix} \nonumber \\ 
&\quad = \begin{bmatrix} \Phi_{T11} \Sigma_{T1}^{-1} & \Phi_{T12} - \Phi_{T11}\Lambda_T + \Lambda_T \Phi_{T22} \\ 0 & Y_T \Phi_{T22} \end{bmatrix}
\\
& \begin{bmatrix} \Sigma_{T1}^{-1} & 0 \\ -\Lambda_T^H & I \end{bmatrix} \Sigma_T \begin{bmatrix} \Sigma_{T1}^{-1} & -\Lambda_T \\ 0 & I \end{bmatrix} 
= \begin{bmatrix} \Sigma_{T1}^{-1} & 0 \\ 0 & Y_T \end{bmatrix}
\end{align*}
we have that \eqref{430475606} is equivalent to the following three equations
\begin{align*}
0 =& \Phi_{T11}^H \Sigma_{T1} \Phi_{T11} - \Sigma_{T1} + \Pi_{T1}^H \Pi_{T1} \\
0 =&  \Phi_{T11}^H \Sigma_{T1} \Phi_{T12} + \Pi_{T1}^H [-\Pi_{T1}\Lambda_T+\Pi_{T2}]
\nonumber \\
& -\Phi_{T11}^H \Sigma_{T1} \Phi_{T11} \Lambda_T + \Phi_{T11}^H \Sigma_{T1} \Lambda_T \Phi_{T22} 
\\
0 =& \Phi_{T22}^H Y_T \Phi_{T22} - Y_T \nonumber \\ & 
+ \left[ -\Pi_{T1}\Lambda_T + \Pi_{T2} \right]^H 
\left[ -\Pi_{T1}\Lambda_T + \Pi_{T2} \right] \nonumber \\
& + \left[ \Phi_{T12} - \Phi_{T11} \Lambda_T + \Lambda_T \Phi_{T22} \right]^H \Sigma_{T1} \nonumber \\ 
& \quad\quad \times \left[ \Phi_{T12} - \Phi_{T11} \Lambda_T + \Lambda_T \Phi_{T22} \right] 
\end{align*}
and 
\begin{align*}
& \gamma_n(x_n)  \nonumber \\ 
&= \tfrac{1}{\tau} \tr\left\{ [\Gamma_{T1}+\Lambda_T\Gamma_{T2}]^H \Sigma_{T1} [\Gamma_{T1}+\Lambda_T\Gamma_{T2}] \right\} 
\nonumber \\ & \quad + \tfrac{1}{\tau} \tr \big\{ [-\Pi_{T1}\Lambda_T+\Pi_{T2}] \Sigma_{T2} 
[-\Pi_{T1}\Lambda_T+\Pi_{T2} ]^H \big\}
\nonumber \\ & \quad + \tfrac{1}{\tau}\tr\big\{ \left[ \Phi_{T12} - \Phi_{T11} \Lambda_T + \Lambda_T \Phi_{T22} \right]^H \Sigma_{T1} 
\nonumber \\ & \quad\quad \times \left[ \Phi_{T12} - \Phi_{T11} \Lambda_T + \Lambda_T \Phi_{T22} \right] \Sigma_{T2} \big\}
\end{align*}
where $\Sigma_{T2}=\Sigma_{T2}^H\succ0$ is the solution to
\begin{equation*} 
0 = \Phi_{T22} \Sigma_{T2} \Phi_{T22}^H - \Sigma_{T2} + \Gamma_{T2}\Gamma_{T2}^H
\end{equation*}
Note that $\Lambda_T$ is affine in $x_n$, because $\Phi_{T12}$ and $\Pi_2$ are.
Meanwhile $\Sigma_{T1}$ and $\Sigma_{T2}$ are invariant on $x_n$.
We conclude that $\gamma_n(x_n)$ as formulated above is quadratic.

Convexity of $\gamma_n(x_n)$ (i.e., the claim that $H_n \succeq 0$) follows from the fact that because $\Phi_{T11}$ and $\Phi_{T22}$ are asymptotically stable (i.e., Schur), resulting in both $\Sigma_{T1}\succeq 0$ and $\Sigma_{T2}\succeq 0$. 

\subsection{Proof of Lemma \ref{KYP_theorem}}

The standard KYP theorem states the necessary and sufficient condition as the existence of $\Sigma_U=\Sigma_U^H\succ 0$ such that
\begin{equation*}
\begin{bmatrix} \Phi_U^H \Sigma_U \Phi_U - \Sigma_U & \Phi_U^H \Sigma_U \Gamma_U - \Pi_U^H \\
\Gamma_U^H \Sigma_U \Phi_U - \Pi_U & \Gamma_U^H \Sigma_U \Gamma_U - \Delta_U - \Delta_U^H \end{bmatrix} \preceq 0
\end{equation*}
where $\Phi_U$, $\Gamma_U$ and $\Pi_U$ are partitioned as in \eqref{495869669}. 
Let 
\begin{align*}
\Sigma_U =& 
\begin{bmatrix} \Sigma_{U1} & 0 \\ -\Lambda_U^H & I \end{bmatrix}^{-1} 
\begin{bmatrix} I & \Lambda_U \\ 0 & \Sigma_{U2} \end{bmatrix}
\end{align*}
Then the above condition is equivalently stated in terms of $\Sigma_{U1}$, $\Sigma_{U2}$, and $\Lambda_U$ as 
\begin{align*}
\begin{bmatrix} -\Sigma_{U1} & \star & \star \\
0 & \Phi_{U22}^H \Sigma_{U2} \Phi_{U22} - \Sigma_{U2} & \star \\
-\Pi_{U1} \Sigma_{U1} & \left( \begin{array}{c} \Gamma_{U2}^H \Sigma_{U2} \Phi_{U22} \\ +\Pi_{U1} \Lambda_U - \Pi_{U2} \end{array} \right) & \left( \begin{array}{c} \Gamma_{U2}^H \Sigma_{U2} \Gamma_{U2} \\ -\Delta_U-\Delta_U^H \end{array} \right) \end{bmatrix} &
\nonumber \\ 
+
\Theta_U^H \Sigma_{U1} \Theta_U
\preceq 0 & 
\end{align*}
where
\begin{equation*}
\Theta_U \triangleq \begin{bmatrix} \Phi_{U11} \Sigma_{U1} & \left( \begin{array}{c} 
-\Phi_{U11} \Lambda_U \\ + \Phi_{U12} \\ + \Lambda_U \Phi_{U22} \end{array} \right)
& \left( \begin{array}{c} \Gamma_{U1} \\ +\Lambda_U \Gamma_{U2} \end{array} \right) \end{bmatrix}
\end{equation*}
and where $\star$ denotes Hermitian symmetry.
Inequality \eqref{convex_pr_lmi} is obtained from the above by taking a Schur complement of the last term on the left-hand side, then taking another Schur complement of the top left block. 

\subsection{Proof of Theorem \ref{459569960707096069059}}

That $\gamma_n^\star$ is monotonically nonincreasing in $n$ is verified by observing that if $\{ \Irk{Q}_{1,n}^\star(0),...., \Irk{Q}_{1,n}^\star(n) \}$ are the optimal Markov parameters for order $n$, then $\{ \Irk{Q}_{1,n}^\star(0),...., \Irk{Q}_{1,n}^\star(n), 0 \}$ is feasible for $n \leftarrow n+1$.  
The uniqueness of the limit $\lim_{n\rightarrow\infty} \gamma_n^\star$ then follows from this result, together with the positive semidefiniteness of $\gamma$.

For parameter $\alpha\in[0,1]$, let $\Tfz{Q}_\alpha$ be defined as
\begin{equation*}
\Tfz{Q}_\alpha = (1-\alpha) \Tfz{Q}^\star + \alpha [\epsilon I + \Tfz{P}_{uy}]^{-1}.
\end{equation*}
With $\alpha = 0$, $\Tfz{Q}_\alpha = \Tfz{Q}^\star$ and therefore satisfies constraint \eqref{Q_constraint} by assumption.
Meanwhile at $\alpha = 1$, $\Tfz{Q}_\alpha = [\epsilon I +\Tfz{P}_{uy}]^{-1}$ which satisfies \eqref{Q_constraint} conservatively, with 
\begin{align*}
&\Tfz{Q}_\alpha(e^{j\Omega}) + \Tfz{Q}_\alpha^H(e^{j\Omega})  \nonumber 
\\
& - \Tfz{Q}_\alpha(e^{j\Omega}) [\epsilon I + \Tfz{P}_{uy}(e^{j\Omega}) + \Tfz{P}_{uy}^H(e^{j\Omega}) ] \Tfz{Q}_\alpha^H(e^{j\Omega}) 
\\
& = \epsilon [\epsilon I + \Tfz{P}_{uy}]^{-1} [\epsilon I + \Tfz{P}_{uy}]^{-H}
\\
&\succeq \nu I
\end{align*}
where $\nu \triangleq \epsilon / \| \epsilon I + \Tfz{P}_{uy} \|_{\Tfz{\Set{H}}_\infty}^2$.
Because \eqref{Q_constraint} is convex, $\Tfz{Q}_\alpha$ satisfies constraint \eqref{Q_constraint} conservatively for all $\alpha \in (0,1]$, with 
\begin{align*}
&\Tfz{Q}_\alpha(e^{j\Omega}) + \Tfz{Q}_\alpha^H(e^{j\Omega})  \nonumber 
\\
& - \Tfz{Q}_\alpha(e^{j\Omega}) [\epsilon I + \Tfz{P}_{uy}(e^{j\Omega}) + \Tfz{P}_{uy}^H(e^{j\Omega}) ] \Tfz{Q}_\alpha^H(e^{j\Omega}) \succeq \nu \alpha I
\end{align*}

The objective value achieved by $\Tfz{Q}_\alpha$, denoted $\gamma_\alpha$, is
\begin{align*}
\gamma_\alpha =& \frac{1}{\tau} \| \Tfz{P}_{fz} - \Tfz{P}_{uz} \Tfz{Q}^\star \Tfz{P}_{fy} + \alpha \Tfz{P}_{uz} [(\epsilon I+\Tfz{P}_{uy})^{-1}-\Tfz{Q}^\star] \Tfz{P}_{fy} \|_{\Tfz{\Set{H}}_2}^2 
\\
\leqslant & \gamma^\star + 
2 \alpha \sqrt{\frac{\gamma^\star}{\tau}} \| \Tfz{P}_{uz} [(\epsilon I+\Tfz{P}_{uy})^{-1}-\Tfz{Q}^\star] \Tfz{P}_{fy} \|_{\Tfz{\Set{H}}_2}
\\ & 
+
\frac{\alpha^2}{\tau} \| \Tfz{P}_{uz} [(\epsilon I+\Tfz{P}_{uy})^{-1}-\Tfz{Q}^\star] \Tfz{P}_{fy} \|_{\Tfz{\Set{H}}_2}^2
\end{align*}
where the triangle inequality has been used.
The norms in the above expression are finite, and consequently there is a nondecreasing function $\alpha(\delta) : \Set{R}_{\geqslant 0} \rightarrow [0,1]$ with $\alpha(0) = 0$, such that $\gamma_{\alpha(\delta)} - \gamma^\star \leqslant \delta$ for all $\delta \in \Set{R}_{\geqslant 0}$.

Now, for a given $\alpha \in [0,1]$, define the associated incremental Youla parameter $\Tfz{Q}_{1,\alpha}$ as
\begin{equation*}
\Tfz{Q}_{1,\alpha} = \Tfz{M}_0^{-1} [\Tfz{Q}_\alpha - \Tfz{Q}_0] \Tfz{N}_0^{-1}
\end{equation*}
Let $\Tfz{Q}_{1,\alpha,n}$ be related to $\Tfz{Q}_{1,\alpha}$ via its Markov coefficients as
\begin{equation*}
\Irk{Q}_{1,\alpha,n}(k) \triangleq \left\{ \begin{array}{lll}
 \Irk{Q}_{1,\alpha}(k) &:& k \in \{0,...,n\} \\
0 & :& k \notin \{0,...,n\}
\end{array}\right.
\end{equation*}
with the associated total Youla parameter then being $\Tfz{Q}_{\alpha,n} \triangleq \Tfz{Q}_0 + \Tfz{M}_0 \Tfz{Q}_{1,\alpha,n} \Tfz{N}_0$.
Because $\Tfz{Q}_\alpha$ satisfies \eqref{Q_constraint} it is known to be in $\Tfz{\Set{H}}_\infty$ and therefore $\Tfz{\Set{H}}_2$.
Because $\Tfz{Q}_\alpha \in \Tfz{\Set{H}}_2$,
$\lim_{n\rightarrow\infty} \gamma(\Tfz{Q}_{\alpha,n}) = \gamma_\alpha$.
Furthermore, there exists a nondecreasing function $\eta : \Set{R}_{\geqslant 0} \rightarrow \Set{Z}_{>0}$ with $\eta(0)=0$, such that each $\beta \in \mathbb{R}_{>0}$, $\| \Tfz{Q}_\alpha - \Tfz{Q}_{\alpha,\eta(\beta)} \|_{\Tfz{\Set{H}}_\infty}^2 \leqslant 1/\beta$.
Constraint \eqref{Q_constraint} is conservatively satisfied for $\Tfz{Q}_{\alpha,\eta(\beta)}$ if 
\begin{align*}
\Tfz{Q}_\alpha(e^{j\Omega}) + \Tfz{Q}_\alpha^H(e^{j\Omega}) 
-\Tfz{Q}_\alpha(e^{j\Omega}) \Tfz{S}^H(e^{j\Omega}) \Tfz{S}(e^{j\Omega}) \Tfz{Q}_\alpha^H(e^{j\Omega}) &
\\
\succeq \left( 2\beta^{-1}+ \| \Tfz{S} \|_{\Tfz{\Set{H}}_\infty}^2 \beta^{-2} \right) I
\end{align*}
where $\Tfz{S}$ is the spectral factor from Theorem \ref{PRLemma}.
We conclude that for each $\delta \in \mathbb{R}_{>0}$, choosing
\begin{equation*}
\beta(\delta) \geqslant 
\frac{\| \Tfz{S} \|_{\Tfz{\Set{H}}_\infty}^2} {-1 + \sqrt{ 1 + \| \Tfz{S} \|_{\Tfz{\Set{H}}_\infty}^2 \nu \alpha(\delta) }}
\end{equation*}
results in $\Tfz{Q}_{\alpha,\eta(\beta(\delta))}$ which is both feasible under constraint \eqref{Q_constraint}, and in the domain of OP4 with $n = \eta(\beta(\delta))$.

We can therefore construct a sequence of values $\{\beta_n, n \in \mathbb{Z}_{>0}\}$, each component of which constitutes the largest value of $\beta \in \mathbb{R}_{\geqslant 0}$ for which $n = \eta(\beta)$.
Similarly we construct sequence $\{\delta_n, n \in \mathbb{Z}_{>0}\}$, each component of which constitutes the smallest value of $\delta \in \mathbb{R}_{\geqslant 0}$ for which $\beta_n = \beta(\delta)$.
Note that this sequence is monotonically decreasing with $\lim_{n\rightarrow\infty} \delta_n = 0$. 
Because $\Tfz{Q}_{\alpha(\delta_n),n}$ is in the feasibility domain for OP4 for each $n \in \Set{Z}_{>0}$, it follows that for each $n \in \mathbb{Z}_{>0}$,
\begin{equation*}
\gamma( \Tfz{Q}_{\alpha(\delta_n),n} ) \geqslant \gamma_n^\star.
\end{equation*}
But 
\begin{align*}
\lim\limits_{n \rightarrow \infty} \gamma( \Tfz{Q}_{\alpha(\delta_n),n} )
= \lim\limits_{\alpha \rightarrow 0} \gamma( \Tfz{Q}_\alpha ) = \gamma^\star
\end{align*}
and because $\gamma_n^\star \geqslant \gamma^\star$ for all $n \in \mathbb{Z}_{>0}$, we conclude that $\gamma_n^\star \rightarrow \gamma^\star$ as $n\rightarrow \infty$.
Convergence of $\Tfz{Q}_n^\star \rightarrow \Tfz{Q}^\star$ as in \eqref{76978696585} follows from the uniqueness of $\Tfz{Q}^\star$ as the minimizer of OP3.

\subsection{Proof of Theorem \ref{lb_converence_theorem}}

That $\check{\gamma}_n^\star \leqslant \gamma^\star$ follows immediately the fact that $\check{\gamma} \leqslant \gamma$ for each $\Tfz{Q}$, and the feasibility domain for OP6 is a relaxation of that of OP3.
Furthermore, because the domain of $\Tfz{V}$ is tightened as $n$ increases, and because the value of $\check{\gamma}$ for any fixed $\Tfz{V}$ increases with $n$, it follows that $\check{\gamma}_n^\star$ is monotonically nondecreasing in $n$, and therefore converges to a well-defined limit.
To show that this limit is $\gamma^\star$, let $\Irk{Q}_n^\star(k), k \in \{0,...,n\}$ be the solution to OP6 for a given $n$, and define
\begin{align*}
\Irk{\Theta}_n(k) \triangleq & \left\{ \begin{array}{lll} \tfrac{n+1-k}{n+1} \Irk{Q}_n^\star(k) &:& k \in \{0,...,n\} \\ 0 &:& k \notin \{0,...,n\} \end{array} \right. \\
\Irk{\Psi}_n(k) \triangleq & \left\{ \begin{array}{lll} \tfrac{n+1-k}{n+1} \Irk{F}(k) &:& k \in \{0,...,n\} \\ 0 &:& k \notin \{0,...,n\} \end{array} \right.
\end{align*}
It follows that 
\begin{align*}
\Tfz{\Theta}_n(e^{j\Omega}) =& \frac{1}{2\pi} \int_{-\pi}^\pi \tfrac{1}{n+1}|\zeta(\Omega)|^2 \Tfz{Q}_n^\star(e^{j\Omega}) d\Omega \\
\Tfz{\Psi}_n(e^{j\Omega}) =& \frac{1}{2\pi} \int_{-\pi}^\pi \tfrac{1}{n+1}|\zeta(\Omega)|^2 \Tfz{F}(e^{j\Omega}) d\Omega
\end{align*}
where $\zeta(\Omega) \triangleq \tfrac{\sin((n+1)\Omega/2)}{\sin(\Omega/2)}$.
Next, note that because $\Tfz{Q}_n^\star$ satisfies feasibility for OP6, it must be that \eqref{V_constraint} holds when convolved with $|\zeta(\Omega)|^2$, and consequently
\begin{equation*}
\begin{bmatrix} \Tfz{\Theta}_n(e^{j\Omega}) & 0 \\ \Tfz{\Theta}_n(e^{j\Omega}) & \Tfz{\Psi}_n(e^{j\Omega}) \end{bmatrix}
+ \begin{bmatrix} \Tfz{\Theta}_n(e^{j\Omega}) & 0 \\ \Tfz{\Theta}_n(e^{j\Omega}) & \Tfz{\Psi}_n(e^{j\Omega}) \end{bmatrix}^H
\succeq 0
\end{equation*}
holds for all $\Omega \in [-\pi,\pi]$.
Also, note that because $\Tfz{F}(e^{j\Omega})$ is continuous in $e^{j\Omega}$ it follows that $\Tfz{\Psi}_n(e^{j\Omega})$ converges pointwise to $\Tfz{F}_n(e^{j\Omega})$ as $n\rightarrow\infty$ for all $\Omega \in [-\pi,\pi]$.
As a consequence, for each $\kappa \in \Set{R}_{>0}$ there exists a $n > 0$ and a $\Tfz{Q}$ satisfying \eqref{Q_constraint}, for which $\| \Tfz{Q} - \Tfz{\Theta}_n \|_{\Tfz{\Set{H}}_2} < \kappa$. 
So if we assume $\lim_{n\rightarrow\infty} \check{\gamma}_n^\star < \gamma^\star$, this implies the existence of a $\Tfz{Q}$ that is vanishingly close to the feasibility domain of OP3, but renders an objective $\gamma$ that is finitely lower than $\gamma^\star$.
But this is impossible because the objective $\gamma$ is continuous in $\Tfz{Q}$, leading to a contradiction.  

%%%%%%%%%%%%%%%%%%%%%%%%%%%%%%%%%%%%%%%%%%%%%%%%%%%%%%%%%%%%%%%%%%%%%%%%%%%%%%%%%%%%%%%%%%%%%%%%%%%%%%%%%%%%%%%%%%%%%%%%%%%%%%%%%%%%%%%%%%

%\bibliography{regen_bib}  
%\bibliographystyle{IEEEtran}   

%%%%%%%%%%%%%%%%%%%%%%%%%%%%%%%%%%%%%%%%%%%%%%%%%%%%%%%%%%%%%%%%%%%%%%%%%%%%%%%%%%%%%%%%%%%%%%%%%%%%%%%%%%%%%%%%%%%%%%%%%%%%%%%%%%%%%%%%%%

\vspace{-25pt}
\begin{IEEEbiography}[{\includegraphics[width=1in,height=1.25in,clip,keepaspectratio]{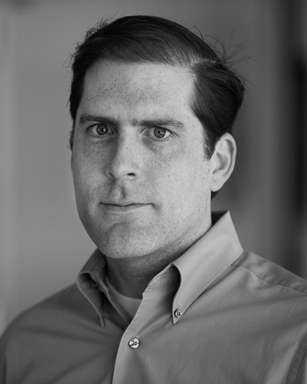}}]{Jeff Scruggs} (M'11) is a Professor in the Departments of Civil and Environmental Engineering, as well as of Electrical and Computer Engineering, at the University of Michigan, which he joined in 2011.  He received his B.S. and M.S. degrees in Electrical Engineering from Virginia Tech in 1997 and 1999, respectively, and his Ph.D. in Applied Mechanics from Caltech in 2004. Prior to joining the University of Michigan, he held postdoctoral positions at Caltech and the University of California, San Diego, and was on the faculty at Duke University from 2007-11. His current research is in the areas of mechanics, vibration, energy, and control.
\end{IEEEbiography}

\end{document}